\documentclass[11pt, a4paper,oneside,reqno]{amsart}
\usepackage[dvips]{epsfig}
\usepackage{amsmath}
\usepackage{amssymb}
\usepackage{amsfonts}
\usepackage{bbm}
\usepackage{amsthm}
\usepackage{amsbsy}
\usepackage{amsgen}
\usepackage{amscd}
\usepackage{amsopn}
\usepackage{amstext}
\usepackage{amsxtra}
\usepackage[utf8]{inputenc}
\usepackage{enumerate}
\usepackage{hyperref}
\usepackage{lineno}
\usepackage{marginnote}
\usepackage{verbatim}
\usepackage{calrsfs}
\usepackage{times, graphicx}
\usepackage{eso-pic}
\usepackage{lastpage}
\usepackage{ulem}
\usepackage{authblk}

\DeclareMathAlphabet{\pazocal}{OMS}{zplm}{m}{n}
\usepackage{amsaddr}

\numberwithin{equation}{section}

\makeatletter
\newcommand*\rel@kern[1]{\kern#1\dimexpr\macc@kerna}
\newcommand*\widebar[1]{%
  \begingroup
  \def\mathaccent##1##2{%
    \rel@kern{0.8}%
    \overline{\rel@kern{-0.8}\macc@nucleus\rel@kern{0.2}}%
    \rel@kern{-0.2}%
  }%
  \macc@depth\@ne
  \let\math@bgroup\@empty \let\math@egroup\macc@set@skewchar
  \mathsurround\z@ \frozen@everymath{\mathgroup\macc@group\relax}%
  \macc@set@skewchar\relax
  \let\mathaccentV\macc@nested@a
  \macc@nested@a\relax111{#1}%
  \endgroup
}
\makeatother

\newtheorem{theorem}{Theorem}[section]
\newtheorem{proposition}[theorem]{Proposition}
\newtheorem{lemma}[theorem]{Lemma}
\newtheorem{corollary}[theorem]{Corollary}

\theoremstyle{definition}

\newtheorem{remark}[theorem]{Remark}

\newcommand{\supp}{\operatorname{supp}}

\newcommand{\area}{\operatorname{Area}}

\newcommand{\Ec}{\pazocal{E}}

\newcommand{\Rc}{\pazocal{R}}

\newcommand{\length}{\operatorname{length}} 
\newcommand{\intinf}{\int_{-\infty}^\infty}

\newcommand {\R} {\mathbb{R}}

\newcommand {\Z} {\mathbb{Z}}

\newcommand{\eigen}{\Lambda} 

\begin{document}

\title{The Robin problem on rectangles}
\date{October 7, 2021}
\begin{abstract}
We study the statistics and the arithmetic properties of the Robin spectrum of a rectangle. A number of results are obtained for the multiplicities in the spectrum, depending on the Diophantine nature of the aspect ratio. In particular, it is shown that for the square, unlike the case of Neumann eigenvalues where there are unbounded multiplicities of arithmetic origin, there are no multiplicities in the Robin spectrum for sufficiently small (but nonzero) Robin parameter except a systematic symmetry. In addition, uniform lower and upper bounds are established for the Robin-Neumann gaps in terms of their limiting mean spacing. Finally, that the pair correlation function of the Robin spectrum on a Diophantine rectangle is shown to be Poissonian.
\end{abstract}

\author{Ze\'ev Rudnick$^{1}$ and Igor Wigman$^{2}$}
\address{$^{1}$School of Mathematical Sciences, Tel Aviv University, Tel Aviv 69978, Israel}
\address{$^{2}$Department of Mathematics, King's College London, UK}
\email[1]{rudnick@tauex.tau.ac.il}
\email[2]{igor.wigman@kcl.ac.uk}


\maketitle

\section{Statement of main results}


Let $\Omega \subset \R^2$ be  a compact planar domain with Lipschitz boundary. The Robin eigenvalue problem on $\Omega$ is to solve the eigenvalue equation $-\Delta f=\lambda f$ with boundary conditions
\begin{equation*}\label{robin bdry cond}
\frac{\partial f}{\partial n}(x)+\sigma f(x)=0,\quad x\in \partial \Omega
\end{equation*}
where $ \frac{\partial f}{\partial n} $ is the derivative in the direction of the outward pointing normal to $\partial \Omega$, and $\sigma>0$.
This boundary condition arises in the study of heat conduction, see e.g. the textbook \cite[Chapter 1]{Strauss}.
Our goal is to study arithmetic properties and statistics of the Robin eigenvalues on a rectangle. For results related to shape optimization for the first two eigenvalues of the Robin Laplacian on a rectangle, see \cite{Laugesen}.

Consider the case of the unit square. For $\sigma=0$, the Neumann eigenvalues on the unit square are explicitly given as $\pi^2( n^2+m^2)$ for integer $n,m\geq 0$. In particular, there are multiplicities coming from the many different ways of writing some of the  integers as a sum of two squares.

For $\sigma\neq 0$ there is no known explicit formula. The problem is however separable, with an orthogonal basis of eigenfunctions of the form $u_{n,m}(x,y) = u_n(x)\cdot u_m(y)$, where $u_n(x)$ are the eigenfunctions of the Laplacian on the unit interval:
$-u_n''=k_n^2 u_n$, satisfying the one-dimensional Robin boundary conditions
\[
-u'(0) +\sigma u(0)=0, \quad u'(1)+\sigma \cdot u(1)=0  .
\]
The frequencies $k_n$ are the unique solutions of the secular equation
\begin{equation}\label{secular eq}
\tan (k_n) =\frac{2\sigma k_n}{k_n^2-\sigma^2}
\end{equation}
in the range $n\pi <k_n<(n+1)\pi$, $n\geq 0$, 
see Lemma \ref{lem:kn sec expand}.
The eigenfunction corresponding to $k_{n}$ is $u_n(x)=k_n\cos(k_nx)+\sigma\sin(k_n x)$.
The Robin eigenvalues on the unit square are
$$\eigen_{n,m}= k_n^2+k_m^2,$$
with eigenfunction $u_n(x)\cdot u_m(y)$, and
admit the symmetry $\eigen_{n,m} = \eigen_{m,n}$. 
For the rectangle $$\Rc_{L}=[0,1]\times [0,L],$$ with the aspect ratio $L>0$,
the Robin energy levels of $\Rc_{L}$ are all numbers
\begin{equation}
\label{eq:lambda_mnL def}
\eigen_{L;n,m}(\sigma)= k_{n}(\sigma)^{2}+\frac{1}{L^{2}}\cdot k_{m}(\sigma\cdot L)^{2} ,\quad  n,m\geq 0.
\end{equation}
Note that if $L\ne 1$, there is no longer the symmetry $(n,m)\mapsto (m,n)$.

\subsection{Multiplicities}

We now consider possible multiplicities in the Robin spectrum of rectangles.
Recall that for the square, and more generally for a rectangle $\Rc_L=[0,1]\times [0,L]$ with $L^2$ rational, the Neumann spectrum has large multiplicities of arithmetic nature, whereas for $L^2$ irrational there are no multiplicities.
Our first goal is to show that for $\sigma>0$ sufficiently small, 
there are no multiplicities in the Robin spectrum of the square beyond the trivial symmetry $\eigen_{1;n,m}=\eigen_{1;m,n}$.

\begin{theorem}
\label{thm:no multiplicities sqr}


There exists $\sigma_{0}>0$ so that for $0<\sigma<\sigma_0$
there are no spectral multiplicities other than the trivial ones $\eigen_{1;n,m}(\sigma)=\eigen_{1;m,n}(\sigma)$.

\end{theorem}

In the proof of Theorem \ref{thm:no multiplicities sqr} (see section \ref{sec:no mult proof}) we shall see that as $\sigma$ varies, the eigenvalues $\eigen_{1;n,m}(\sigma)$ evolve at different rates, depending on $n,m$.
These discrepancies are sufficiently large to break the degeneracies of the Neumann case ($\sigma=0$) for $\sigma>0$ sufficiently small.

One should compare the statement of Theorem \ref{thm:no multiplicities sqr} asserting that, for $\sigma>0$ sufficiently small, the Robin spectrum of the square is non-degenerate, to the recent result \cite{RW Quebec} asserting that the Robin spectrum of the hemisphere is non-degenerate for every $\sigma>0$. On the other hand, the Robin spectrum of the square does admit nontrivial spectral degeneracies
for sufficiently large $\sigma$ (see Proposition~\ref{prop:mult (3,4),(1,5)}).

\vspace{2mm}

Next, we consider the rectangle $\Rc_{L}$ with $L^2$ irrational. Unlike the square, here there exist multiplicities even for small $\sigma$:

\begin{theorem}\label{thm:no multiplicities rectangle}
If $L^2$ is irrational, then there are arbitrarily small $\sigma>0$ for which there are multiplicities in the Robin spectrum of the rectangle
$\Rc_L$.
\end{theorem}

The proof of Theorem~\ref{thm:no multiplicities rectangle} involves some arithmetic, in particular, in showing
that the set of values attained by the indefinite ternary quadratic form
$$ Q(x,y,z) =   L^2 x^{2} +y^{2}-z^{2}
$$
at integer values of $(x,y,z)$ intersects every neighbourhood of the origin: $-\epsilon <Q(n,m,m')<0$ with all variables nonzero integers. This is a variation on the Oppenheim conjecture (proved by Margulis \cite{Margulis}), which turns out to admit a simple solution using only the density of the fractional parts of $L^2 n^2 \bmod 1$, due to Hardy and Littlewood \cite{HL}.

\vspace{2mm}

We next show that in some special cases we can give an upper bound for the multiplicities, and for the number of eigenvalues that are not simple.  Let $\lambda_1(\sigma)\leq \lambda_2(\sigma)\leq \dots $ be the ordering (with multiplicities) of the Robin eigenvalues of $\Rc_{L}$. By
Weyl's law, the number of eigenvalues  of size at most $\lambda$ is asymptotically
\begin{equation}
\label{eq:N spec func def}
N(\lambda)=N_{L;\sigma}(\lambda):=\#\{\lambda_j(\sigma)\leq \lambda\} \sim \frac{\area(\Rc_L)}{4\pi} \lambda, \quad \lambda\to \infty .
\end{equation}
Denote by $N^{\rm mult}(\lambda)$ the number of multiple eigenvalues $\le \lambda$ (again, counting the multiplicities in).

\begin{theorem}
\label{thm:bnded deg bad approx}

If $L^2$ is badly approximable, then there exists $\sigma_{0}>0$ so that for $\sigma<\sigma_0$
all the multiplicities in the Robin spectrum of the rectangle $\Rc_L$ are bounded by $3$, and
\begin{equation}
\label{eq:Nmult << sqrt(lambda)}
N^{\rm mult}(\lambda) \ll \sqrt{\lambda}.
\end{equation}
If, in addition, $L$ is badly approximable then the multiplicities are bounded by $2$.
\end{theorem}

Recall that a number $\theta$ is ``badly approximable'' if there is some $c=c(\theta)>0$ so that for all integer $p,q\in \Z$ with $q\geq 1$, we have
\begin{equation}
\label{eq:badly approx def}
\left|\theta-\frac{p}{q}\right|\geq \frac{c}{q^2}.
\end{equation}
For instance, quadratic irrationalities are badly approximable.

\subsection{Robin-Neumann gaps}

We next turn to the differences between the Robin and Neumann eigenvalues,
or simply RN gaps, introduced in ~\cite{RWY}. Let $\lambda_1(\sigma)=\lambda_{1}^{L}(\sigma)\leq \lambda_2(\sigma)=\lambda_{2}^{L}(\sigma)\leq \dots $ be
the ordering (with multiplicities) of the Robin eigenvalues of the rectangle $\Rc_{L}$.
For example, for $\sigma=0$ we recover the Neumann eigenvalues, and $\lambda_j(\infty)$ are the Dirichlet eigenvalues.
The RN gaps are the nonnegative numbers $$d_{j}=d_{j}^{L}(\sigma):=\lambda_{j}^{L}(\sigma)-\lambda_{j}^{L}(0).$$

For every bounded domain $\Omega$ with piecewise smooth boundary, it was shown in ~\cite[Theorem 1.1]{RWY} that for $\sigma>0$, there is a limiting  mean RN gap, asymptotic to
\begin{equation}
\label{eq:d(sigma) def}
\overline{d}(\sigma):= \lim\limits_{N\rightarrow\infty}\frac{1}{N}\sum\limits_{j=1}^{N}d_{j} (\sigma) =  \frac{2\length(\partial\Omega)}{\area(\Omega)} \cdot \sigma.
\end{equation}
Concerning the individual RN gaps, \cite{RWY} gave a uniform lower bound for arbitrary star-shaped domains with smooth boundary.
For an upper bound, they proved that $$d_{j}(\sigma)\le C_{\Omega}(\lambda_{j}(\infty))^{1/3}\cdot \sigma,$$ valid for any $\Omega$ with smooth boundary. This could be compared to the bound ~\cite[Theorem 2]{Fl} $$0\le d_{j}(\infty)-d_{j}(\sigma) \le C \sigma^{-1/2}\lambda_{j}(\infty)^{2}$$ with $C>0$ absolute, for the distance between the Robin eigenvalues and the corresponding Dirichlet eigenvalue, in the regime $\sigma\rightarrow+\infty$, also assuming that $\Omega$ has a smooth boundary.

For the rectangle ~\cite[Theorem 1.3, Theorem 1.7]{RWY} gave the more precise upper bound
\begin{equation}
\label{eq:RN gaps bounds RWY}
d_{j}^{L}(\sigma) \le C_{L,\sigma},
\end{equation}
for some constant $C_{L,\sigma}>0$.
Here we give uniform upper and lower bounds for the rectangle in terms of the mean gap $\overline{d}(\sigma)$, in particular refining the upper bound \eqref{eq:RN gaps bounds RWY}:

\begin{theorem}
\label{thm:RN gaps bnded}

There exist absolute constants $C>c>0$, so that for every rectangle, for all $\sigma >0$ and $j\ge 1$,
\begin{equation}
\label{eq:dn<Csig}
d_{j}(\sigma) \le  C\cdot \overline{d}(\sigma),
\end{equation}
and for all $\sigma\in (0,1]$,
\begin{equation}
\label{eq:dn>csig}
d_{j}(\sigma) \ge  c\cdot  \overline{d}(\sigma) .
\end{equation}

\end{theorem}


Note that \eqref{eq:dn>csig} can only be valid for
\[
\sigma\leq \frac{1}{c} \frac{\pi^2(L+\frac 1L)}{4(1+L)}.
\]
Indeed,  $\lambda_{1}(0)\le \lambda_{1}(\sigma)\le \lambda_{1}(\infty)$, so that
\[
c\cdot \frac{4(1+L)}{L}\cdot \sigma = c\cdot  \overline{d}(\sigma)\leq d_{1}(\sigma) \le \lambda_{1}(\infty) - \lambda_{1}(0)=\left(1+\frac{1}{L^{2}}\right)\pi^2.
\]


\subsection{Pair-correlation for the Robin energies}
We next turn to study the statistics of the eigenvalues on the scale of their mean spacing.
In our case,
the mean spacing between the eigenvalues is constant
\begin{equation}
\label{eq:sbar mean spac}
\bar s:=\lim\limits_{N \to \infty}\frac 1N \sum_{k\leq N} \Big(\lambda_{k+1}(\sigma)-\lambda_k(\sigma)\Big) \sim \frac{4\pi}{\area \Rc_L}
\end{equation}
by Weyl's law \eqref{eq:N spec func def}.
One popular local statistic is the distribution $P(s)$ of nearest-neighbour gaps $(\lambda_{k+1}(\sigma)-\lambda_k(\sigma))/\bar s$.
For the square (and more generally when $L^2$ is rational),  $P(s)$ is a delta function at the origin \cite{RWY}. However, we expect that if the squared aspect ratio $L^2$ is a Diophantine irrationality, that is there is some $\kappa>0$ so that  $|L^2-p/q|>1/q^\kappa$ for all integers $q> 1$ and $p$, then the nearest neighbour gap distribution will be Poissonian:  $P(s)=e^{-s}$, that is as for uncorrelated levels, cf. \cite{BerryTabor, Marklof2000, RudnickQC}. However at present this quantity is not accessible. A more tractable statistic is the pair correlation function, defined as follows:
For a test function $f\in C_c^\infty(\R)$, we set
\[
R_2^\sigma(f,N) = \frac 1N\sum_{1\leq  k\neq k'\leq N} f\left(\frac{ \lambda_k(\sigma) - \lambda_{k'}(\sigma)}{\bar s}\right) .
\]
The Poisson expectation is that
\begin{equation*}
\label{eq:Poisson exp}
\lim_{N\to \infty} R_2^\sigma(f,N)  =\intinf f(x)dx .
\end{equation*}

\begin{theorem} \label{thm:PC diophantine}
Assume that $L^2$ is a Diophantine irrationality. Then for every fixed $\sigma>0$, the pair correlation function is Poissonian.
\end{theorem}

To prove Theorem~\ref{thm:PC diophantine}, we establish a comparison with the pair correlation of the Neumann spectrum (Proposition~\ref{prop:comparison}),
which was shown to be Poissonian in the Diophantine case by Eskin, Margulis and Mozes \cite{EMM}. There are two key ingredients in the comparison argument: A stronger, asymptotic, form of the bound for the RN gaps of Theorem \ref{thm:RN gaps bnded}
(see Proposition \ref{prop:gap constant ae} below),
and a count of lattice points in annular regions, for which it suffices to appeal to a classical remainder term in the lattice point problem.

\vspace{2mm}

It is of interest to investigate analogues of our results for the case when the boundary conditions are non-constant, that is
$$ \frac{\partial f}{\partial n}(x)+\sigma(x) f(x)=0$$ for $x$ on the boundary, where $\sigma(x)>0$ is a continuous function on the boundary. New methods will be required since we make heavy use of the fact that $\sigma$ is constant.


\section{The one-dimensional problem}
\label{sec:1d problem}

In this section we review some classical properties of the Robin eigenvalue problem on an interval, see e.g. \cite[\S 4.3]{Strauss}.

\subsection{The secular equation}
\label{sec:sec equation}

Let $I=[-\frac 12,\frac 12]$ be the unit interval, and consider the Helmholtz equation
\begin{equation}
\label{eq:Helmholts eq}
f''+ k^2 f = 0,
\end{equation}
subject to Robin boundary conditions
\[
\sigma f\left(-\frac 12\right)-f'\left(-\frac 12\right) = \sigma f\left(\frac 12\right)+f'\left(\frac 12\right)=0.
\]
We use the symmetry $x\mapsto -x$, which is respected by both the second derivative operator $f\mapsto f''$ and the boundary conditions, to separate solutions into even and odd symmetry classes. The even solutions of the eigenvalue equation \eqref{eq:Helmholts eq} are
$f(t) = \cos(kt)$, which, inserting into the boundary conditions gives
\[
k\cdot\tan \left(\frac k2\right) = \sigma .
\]
Likewise, the odd solutions of \eqref{eq:Helmholts eq} are $f(t) = \sin(kt)$, and the secular equation is
\[
-k\cdot\cot\left( \frac k2\right) = \sigma .
\]

As we shall see below, the solutions of the even and odd secular equations interlace, and the totality of solutions
$\{k_n(\sigma):n=0,1,2\dots\} $ are the solutions of the combined secular equation
$$\left(k\tan \left(\frac k2\right)-\sigma\right)\cdot \left(k\cot \left(\frac k2\right)+ \sigma\right)=0,$$ that, after some algebra, reads
\begin{equation}
\label{eq:secular gen R}
\tan(k) = \frac{2\sigma k}{k^{2}-\sigma^{2}}.
\end{equation}
We could also deduce the equation \eqref{eq:secular gen R} directly if we ignore the symmetry $x\mapsto –x$.

\subsection{General intervals}

\label{sec:gen int}

Instead of the unit interval we consider an interval of length $L$. The Laplace eigenfunctions $f''+k^2f=0$
on $[-\frac L2,\frac L2]$ are subject to the Robin boundary conditions
\begin{equation*}
\sigma f\left(-\frac L2\right)-f'\left(-\frac L2\right) = \sigma f\left(\frac L2\right)+f'\left(\frac L2\right)=0.
\end{equation*}
We obtain solutions to the Helmholtz equation on $\left[-\frac L2,\frac L2\right]$ by scaling the corresponding solutions on the unit interval: If $g$ on $\left[-\frac 12,\frac 12\right]$ solves
  $ g'' + k^2  g=0$ and $\sigma g\left(-\frac 12\right)-g'\left(-\frac 12\right)=\sigma g\left(\frac 12\right)+g'\left(\frac 12\right)=0$,
then $f_L(t)=g(t /L)$ on $\left[-\frac L2,\frac L2\right]$ satisfies
$$
f'' + \left(\frac kL\right)^2  f=0, \quad \frac \sigma L f_L\left(-\frac L2\right)-f_L'\left(-\frac L2\right)=0=\frac \sigma L f_L\left(\frac L2\right)+f_L'\left(\frac L2\right).
$$
Hence if we define $k_{L;n}(\sigma):= \frac{1}{L} k_{n}(\sigma\cdot L)$,
then the Robin energy levels on $[0,L]$ are
$$ k_{L;n}(\sigma)^{2} = \frac{1}{L^{2}} (k_{n}(\sigma\cdot L))^{2} ,\quad n\ge 0 .
$$
Note that the secular equation on $\left[-\frac L2,\frac L2\right]$ becomes
\begin{equation*}
\tan(Lk) = \frac{2\sigma k}{k^{2}-\sigma^{2}}.
\end{equation*}

\subsection{Properties of $k_{n}(\sigma)$}
\label{sec:kn(R) analytic}

\begin{lemma}
\label{lem:kn sec expand}

For every $n\ge 0$ and $\sigma\ge 0$ there is a unique solution $k_{n}(\sigma)$ to
the secular equation \eqref{eq:secular gen R} in the range $k_{n}(\sigma)\in  [n\pi ,(n+1)\pi]$.
The functions $\sigma\mapsto k_{n}$ satisfy:

\begin{enumerate}[a.]

\item For all $n\ge 0$, $k_{n}(\cdot)$ are strictly increasing everywhere on $[0,+\infty)$, with
\begin{equation}
\label{eq:kn(0)=pi n}
k_{n}(0) = n\cdot \pi,
\end{equation}
and further
\begin{equation*}
\lim\limits_{\sigma\rightarrow\infty}k_{n}(\sigma)= (n+1)\cdot \pi.
\end{equation*}

\item For $n\ge 1$, the function $\sigma\mapsto k_{n}(\sigma)$ is analytic
everywhere.
Further, for $\sigma< (n+1/2)\pi$, $k_{n}(\sigma) \in [n \pi,(n+1/2)\pi]$, and for $\sigma\ge (n+1/2)\pi$,
$k_{n}(\sigma) \in [(n+1/2)\pi, (n+1)\pi]$. Moreover, $k_{n}(\sigma) = (n+1/2)\pi$ if and only if $\sigma = (n+1/2)\pi$.

\item The function $k_{0}(\cdot)$ is analytic everywhere except
at $(\sigma,k_{0})=(0,0)$.
Further, for $\sigma\in (0,\pi/2)$, $k_{0}(\sigma)>\sigma$, and
\begin{equation}
\label{eq:k0 asympt 2 terms}
k_{0}(\sigma)^{2} = 2\sigma+O(\sigma^{2}).
\end{equation}

\end{enumerate}

\end{lemma}

\begin{proof}

We first consider the odd part of the spectrum: note that the function $$S_-(k):=-k\cdot\cot\left( \frac k2 \right)$$ is even, and for $k\geq 0$ vanishes at $\pi,3\pi,\dots, (2n+1)\pi,\dots $, $n\geq 0$, has singularities at $$k=2\pi, 4\pi,\dots,2n\pi,\dots,$$ $n\geq 1$, and increasing monotonically for $k\geq 0$ between the singularities, because it has positive derivative there: $$S_-'(k) = \frac{k-\sin k}{2 \sin^2(k/2)},$$
see Figure~\ref{fig:seceqev}.
Thus for $\sigma>0$ there is a unique solution $k_{2n-1}(\sigma)$ of $S_+(k)=\sigma$ in each interval $((2n-1)\pi, 2n\pi)$, $n=1,2,\dots$. Moreover, by the analytic implicit function theorem, the solutions $k_{2n-1}(\sigma)$ are analytic in $\sigma$ for $n\geq 1$.

For the even part of the spectrum: The function $$S_+(k):=k\cdot\tan \left(\frac k2 \right)$$ is even, and for $k\geq 0$ vanishes at $0,2\pi,\dots, 2n\pi,\dots $, $n\geq 0$, has singularities  at $k=\pi, 3\pi,\dots,(2n+1)\pi,\dots$, $n\geq 0$, and increasing monotonically for $k\geq 0$ between the singularities, because it has positive derivative there: $$S_+'(k) = \frac{k+\sin k}{2 \cos^2(k/2)},$$
see Figure~\ref{fig:seceqev}. Thus, for $\sigma>0$ there is a unique solution $k_{2n}(\sigma)$ of $S_+(k)=\sigma$ in each interval $(2n\pi, (2n+1)\pi)$. Moreover, by the analytic implicit function theorem, the solutions $k_{2n}(\sigma)$ are analytic in $\sigma$ for $n\geq 1$.
 \begin{figure}[ht]
\begin{center}
\includegraphics[height=50mm]{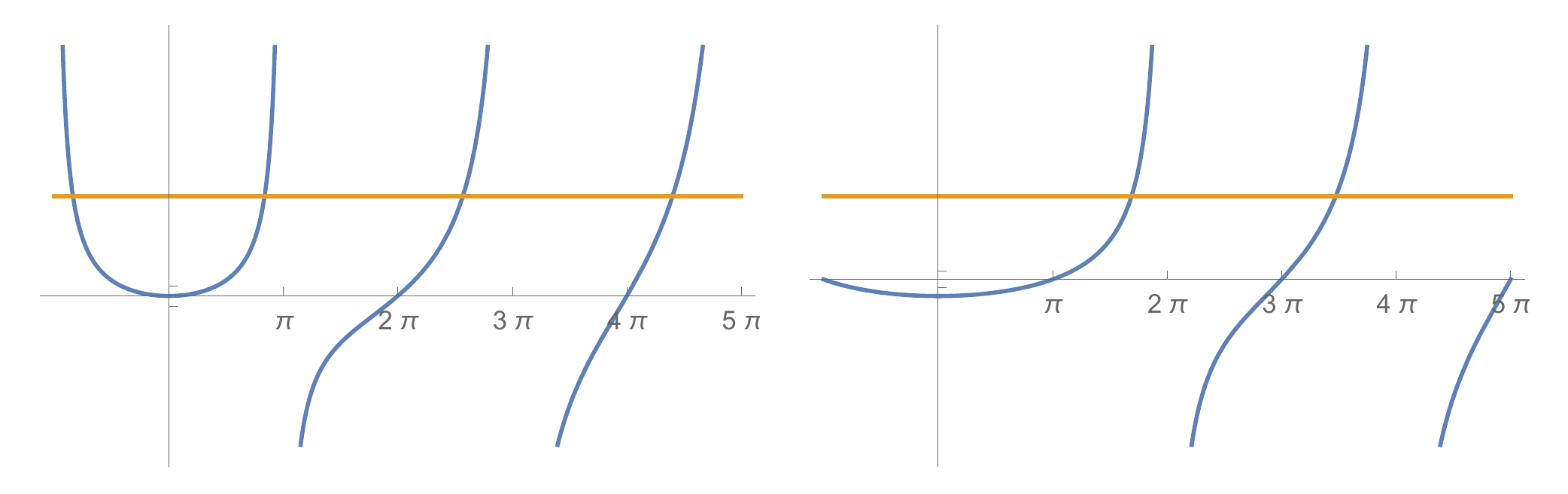}
\caption{Left: The even secular equation $k\tan\left(\frac k2\right) = \sigma$. Right: The odd secular equation $-k\cot\frac k2 = \sigma$.  }
\label{fig:seceqev}
\end{center}
\end{figure}

To see that $k_0(\sigma)>\sigma$ for $\sigma\in (0,\pi/2)$, just note that $0<\tan\left( \frac k2\right)<1$ for $k\in (0,\frac \pi 2)$ so that
$\sigma = k\tan \left(\frac {k_0(\sigma)}2\right)<k_0(\sigma) \cdot 1$ in this range.
Finally, to see \eqref{eq:k0 asympt 2 terms}, we expand, using $k_0(\sigma)\to 0$ as $\sigma\to 0$,
\[
\sigma = k\tan \left(\frac k2\right) = k\left(\frac k2 +O(k^3)\right) = \frac {k^2}{2} +O(k^4)
\]
from which \eqref{eq:k0 asympt 2 terms} follows.
\end{proof}

\subsection{Auxiliary computations}

\begin{lemma}
\label{lem:deriv eval}
For $n\ge 1$ the functions $k_{n}(\cdot)$ satisfy:

\begin{enumerate}[a.]

\item \begin{equation}
\label{eq:kn 1st der comp}
k_{n}'(0)=\frac{2}{\pi n}.
\end{equation}

\item \begin{equation}
\label{eq:k'' orig expr}
(k_{n}(\sigma)^{2})''|_{\sigma=0} = 2(k_{n}(\sigma)\cdot k_{n}'(\sigma))'|_{\sigma=0}= -\frac{8}{(\pi n)^{2}}.
\end{equation}

\item
Uniformly for $n\ge 1$, $0\le \sigma\le 1$, one has
\begin{equation}
\label{eq:k'/2k asymp}
k_n' = \frac{2}{k_n} \cdot \left( 1 +f_{2}(\sigma)\cdot \frac{1}{k_n^{2}}  \right) + \Ec_{n}(\sigma),
\end{equation}
with $f_2(\sigma)=-\sigma(2+\sigma)$, and
\begin{equation}
\label{eq:Ec error bnd}
|\Ec_{n}(\sigma)| = O\left(\frac{\sigma^{2}}{n^{5}}\right).
\end{equation}

\end{enumerate}

\end{lemma}

\begin{proof}
We treat the even secular equation, the odd case is completely analogous. From $$S_+(k):=k\tan\left( \frac k2\right) = \sigma$$ we obtain, by implicit differentiation, $k'=1/S_+'(k)$. Now
\begin{equation}\label{deriv of S_+}
S_+'(k)=\tan\left( \frac k2\right) +\frac{k}{2\cos^2 \left(\frac k2\right)} = \tan\left( \frac k2\right) +\frac k2\left(1+\tan^2 \left(\frac k2\right) \right) .
\end{equation}
Substituting $k_{2n}(0) = 2n\pi$, we obtain $S_+'(k_{2n}(0)) = n\pi = (2n\pi)/2$ which gives \eqref{eq:kn 1st der comp} in the even case.

To obtain \eqref{eq:k'' orig expr}, we use
\begin{equation}\label{expand k''}
(k_n^2)'' = 2(k_n')^2 +2k_nk_n''
\end{equation}
and $k_{2n}' = 1/S_+'$ so that
\[
k_{2n}'' = \left(\frac 1{S_+'}\right)' = -\frac{k_{2n}' S_+''}{(S_+')^2}  = -(k_{2n}')^3 S_+'' .
\]
A computation shows that
\[
S_+'' = \frac{2+k\tan \left(\frac k2\right)}{2\cos^2 \left(\frac k2\right)}
\]
Evaluating at $\sigma=0$, where $k_{2n}(0) = 2n\pi$, we obtain $S_+''(k_{2n}(0)) = 1$ and
\[
k_{2n}''(0) = -k_{2n}'(0)^3S_+''(k_{2n}(0))  = -\left(\frac 2{2n \pi}\right)^3.
\]
Substituting in \eqref{expand k''} with $k_{2n}(0) = 2n\pi$, $k_{2n}'(0) = 2/(2n\pi)$  we deduce \eqref{eq:k'' orig expr}.

To obtain \eqref{eq:k'/2k asymp}, we return to \eqref{deriv of S_+}, use the secular equation to write $\tan \left( \frac k2 \right)= \frac \sigma k$
for $\sigma>0$ and obtain
\[
S_+'(k)= \frac \sigma k +\frac k2 \left( 1 +\frac{\sigma^2}{k^2} \right) = \frac k2 \cdot \left(1-\frac{f_2(\sigma)}{k^2} \right) .
\]
Hence
\[
k' = \frac 1{S_+'(k)} = \frac 2k\left(1+\frac{f_2(\sigma)}{k^2} \right)  +O\left(\frac{f_2(\sigma)^2}{k^5} \right)
\]
which, for $\sigma\le 1$, is \eqref{eq:k'/2k asymp}.

\end{proof}

\section{Spectral degeneracies for the square: proof of Theorem \ref{thm:no multiplicities sqr}}
\label{sec:no mult proof}

\subsection{No multiplicities for the square near $\sigma=0$: Proof of Theorem \ref{thm:no multiplicities sqr}}

A lattice point $(n,m)\in\Z_{\ge 0}^{2}$ gives rise
to the energy (see \eqref{eq:lambda_mnL def}):
\begin{equation}\label{eq:Gn,m as g-}
\eigen_{n,m}(\sigma)= \eigen_{1;n,m}(\sigma):=   k_{n}(\sigma)^{2}+k_{m}(\sigma)^{2}.
\end{equation}
In this section we will use the shorthand $\eigen_{n,m}(\cdot)=\eigen_{1;n,m}(\cdot)$.
For  the proof of Theorem~\ref{thm:no multiplicities sqr} we will need the following propositions.

\begin{proposition}
\label{prop:Gnm' expand}
Uniformly for $n,m\ge 1$ and $0\le \sigma \le 1$, one has
\begin{equation}
\label{eq:Gnm asymp nm}
\eigen_{n,m}'(\sigma)=4\left(2+\frac{f_{2}(\sigma)}{\pi^{2}}\left(\frac{1}{n^{2}}+\frac{1}{m^{2}} \right)\right)+E_{(n,m)}(\sigma)
\end{equation}
with
\begin{equation}
\label{eq:f2 def}
f_{2}(\sigma):=-\sigma(2+\sigma)
\end{equation}
and where the error term satisfies
\begin{equation}
\label{eq:err term Gnm}
|E_{(n,m)}(\sigma)| = O\left(\sigma^{2}\left(\frac{1}{n^{4}} + \frac{1}{m^{4}} \right)\right).
\end{equation}
\end{proposition}

\begin{proof}[Proof of Proposition \ref{prop:Gnm' expand}]

This is a direct conclusion of Lemma \ref{lem:deriv eval}(d), except that we have to justify substituting, up to admissible error term, $\pi n$ and $\pi m$ instead of $k_{n}$ and $k_{m}$ respectively on the r.h.s. of \eqref{eq:k'/2k asymp}.
Indeed,
\begin{equation}
\label{eq:Gnm der as k'}
\eigen_{n,m}'(\sigma)=2\left(k_{n}'(\sigma)k_{n}(\sigma)+k_{m}'(\sigma)k_{m}(\sigma)\right),
\end{equation}
and by virtue of Lemma \ref{lem:deriv eval}(d), we have
\begin{equation*}
k_n' = \frac{2}{k_n} \cdot \left( 1 +f_{2}(\sigma)\cdot \frac{1}{k_n^{2}}  \right) + \Ec_{n}(\sigma),
\end{equation*}
with the error term bounded by \eqref{eq:Ec error bnd},
and hence
\begin{equation}
\label{eq:kk' asympt}
k_n\cdot k_n' = 2 \left( 1 +f_{2}(\sigma)\cdot \frac{1}{k_n^{2}}  \right) + E_{n}(\sigma),
\end{equation}
where
\begin{equation*}
|E_{n}(\sigma)| = O\left( \frac{\sigma^{2}}{n^{4}}  \right) .
\end{equation*}
By the secular equation \eqref{eq:secular gen R},
\begin{equation*}
k_{n}(\sigma)=\pi n + O(\sigma/n),
\end{equation*}
and hence
\begin{equation*}
\frac{1}{k_{n}(\sigma)} = \frac{1}{(\pi n)(1+O(\sigma/n^{2}))}= \frac{1}{\pi n}+O\left(\frac{\sigma}{n^{3}}\right),
\end{equation*}
and
\begin{equation*}
\frac{1}{k_{n}(\sigma)^{2}} = \frac{1}{\pi^{2} n^{2}}+O\left(\frac{\sigma}{n^{4}}\right).
\end{equation*}
Substituting  into \eqref{eq:kk' asympt} produces, after multiplication by $f_{2}(\sigma)$,
an error term of $O\left(\frac{\sigma^{2}}{n^{4}}\right)$ that can be absorbed into $E_{n}(\sigma)$, so that
\eqref{eq:kk' asympt} reads
\begin{equation*}
k_n\cdot k_n' = 2 \left( 1 +f_{2}(\sigma)\cdot \frac{1}{(\pi n)^{2}}  \right) + E_{n}(\sigma).
\end{equation*}
The main statement \eqref{eq:Gnm asymp nm} with the prescribed error term \eqref{eq:err term Gnm}
of Proposition \ref{prop:Gnm' expand} finally follows upon substituting the latter estimate
corresponding to $n$ and $m$ into \eqref{eq:Gnm der as k'}.
\end{proof}

\begin{proposition}
\label{prop:error majorized}

\begin{enumerate}[a.]

\item For all $(n,m)$ and $(n',m')$ so that $n,m,n',m'\ge 1$ and $n^{2}+m^{2}=n'^{2}+m'^{2}$, we have
\begin{equation}
\label{eq:1/n^2+1/m^2>1/n'^2+1/m'^2}
\frac{1}{n^{2}} + \frac{1}{m^{2}}  >
\frac{1}{n'^{2}} + \frac{1}{m'^{2}} \quad   \Longleftrightarrow  \quad nm<n'm'.
\end{equation}

\item
Uniformly for all $(n,m)$, $(n',m')$ so that $n,m,n',m'\ge 1$ and $n^{2}+m^{2}=n'^{2}+m'^{2}$, $nm<n'm'$,
\begin{equation}
\label{eq:1/n^2+1/m^2-1/n'^2+1/m'^2 large}
\frac{1}{n^{4}}+\frac{1}{m^{4}}+\frac{1}{n'^{4}}+\frac{1}{m'^{4}}=
O\left( \left(\left(\frac{1}{n^{2}} + \frac{1}{m^{2}}\right) -\left(\frac{1}{n'^{2}} + \frac{1}{m'^{2}}\right) \right) \right).
\end{equation}
Note that the r.h.s. of \eqref{eq:1/n^2+1/m^2-1/n'^2+1/m'^2 large}
is positive, by \eqref{eq:1/n^2+1/m^2>1/n'^2+1/m'^2}.

\item As $\sigma\rightarrow 0$, uniformly for all $(n,m)$ and $(n',m')$ so that $n,m,n',m'\ge 1$, $n^{2}+m^{2}=n'^{2}+m'^{2}$ and $n'm'>nm$,
\begin{equation}
\label{eq:Enm+En'm'<f4*diff}
 |E_{(n,m)}(\sigma)|+|E_{(n',m')}(\sigma)| = o_{\sigma\rightarrow 0}
 \left(f_{2}(\sigma)\cdot \left(\left(\frac{1}{n'^{2}} + \frac{1}{m'^{2}}\right) - \left(\frac{1}{n^{2}} + \frac{1}{m^{2}}\right)  \right)\right),
\end{equation}
with the r.h.s. of \eqref{eq:Enm+En'm'<f4*diff} positive by part (a) and \eqref{eq:f2 def}.

\end{enumerate}

\end{proposition}

\begin{proof}[Proof of Proposition \ref{prop:error majorized}]

The first statement \eqref{eq:1/n^2+1/m^2>1/n'^2+1/m'^2} of Proposition \ref{prop:error majorized} is straightforward.
For the second one \eqref{eq:1/n^2+1/m^2-1/n'^2+1/m'^2 large} we denote $K:=n^{2}+m^{2}=n'^{2}+m'^{2}$,
and choose any parameter $0<\epsilon<1$ sufficiently small, whose precise value is irrelevant, except that it will be fixed throughout this proof.
We further assume w.l.o.g. that $n\le m$ and $n'\le m'$
(and $nm<n'm'$), implying in particular that
\begin{equation}
\label{eq:n/sqrt(K)<=2}
n,n'\le \sqrt{\frac{K}{2}} \text{ and } m,m'\ge \sqrt{\frac{K}{2}},
\end{equation}
and $n<n'$.
We write
\begin{equation}
\label{eq:int not vanish=> >=1}
\begin{split}
&\left(\frac{1}{n^{2}}+\frac{1}{m^{2}} - \frac{1}{n'^{2}}-\frac{1}{m'^{2}}\right)
= \left(\frac{K}{n^{2}m^{2}}- \frac{K}{n'^{2}m'^{2}}\right)
\\&=\frac{K}{n^{2}m^{2}n'^{2}m'^{2}}(n'^{2}m'^{2}-n^{2}m^{2})
=\frac{K(nm+n'm')}{n^{2}m^{2}n'^{2}m'^{2}}(n'm'-nm).
\end{split}
\end{equation}

First, assume that both $n,n'>\epsilon \sqrt{K}$. Then, using the trivial bound $n'm'-nm\ge 1$ in
\eqref{eq:int not vanish=> >=1} yields
\begin{equation}
\label{eq:diff>>sum eps}
\begin{split}
&\left(\frac{1}{n^{2}}+\frac{1}{m^{2}} - \frac{1}{n'^{2}}-\frac{1}{m'^{2}}\right)
\ge
\frac{1}{n^{2}n'm'} \gg \frac{\epsilon^{-2}}{n^{4}} \gg \epsilon^{-2}\cdot \left(\frac{1}{n^{4}}+
\frac{1}{m^{4}} + \frac{1}{n'^{4}} + \frac{1}{m'^{4}} \right).
\end{split}
\end{equation}
Otherwise, we assume (w.l.o.g. thanks to the above assumptions) that $n\le \epsilon\sqrt{K}$.
In this case we
can improve upon the trivial lower bound $n'm'-nm\ge 1$ in the following way.

Define
\begin{equation*}
f_{K}(n):= n\cdot \sqrt{K-n^{2}} = K\cdot g(n/\sqrt{K}),
\end{equation*}
where $g(y):=y\cdot \sqrt{1-y^{2}}$ on $y\in [0,1]$ (in fact, in our context, $y\in [0,1/\sqrt{2}]$, see \eqref{eq:n/sqrt(K)<=2}),
and, under the assumptions above, if $n<\epsilon \sqrt{K} $, we have
\begin{equation*}
\begin{split}
n'm'-nm = f_{K}(n')-f_{K}(n) = K\left( g(n'/\sqrt{K}) - g(n/\sqrt{K})   \right)\geq  \frac 12 \sqrt{K}(n'-n)>0,
\end{split}
\end{equation*}
and claim that, assuming that $\epsilon>0$ is sufficiently small (recall that $n\le\epsilon\sqrt{K}$),
\begin{equation}
\label{eq:g diff >> sqrt(K)(n'-n)}
g(n'/\sqrt{K}) - g(n/\sqrt{K})  \ge \frac{1}{10}\sqrt{K}(n'-n)
\end{equation}
so that
\begin{equation}
\label{eq:n'm'-nm nontriv bound}
n'm'-nm \gg \sqrt{K}(n'-n)>0
\end{equation}
improves on the trivial bound.

Indeed, by the mean value theorem, for some $\xi  \in \left(\frac {n}{\sqrt{K}}, \frac{n'}{\sqrt{K}}\right)$,
\begin{equation}
\label{eq:MVT}
 g(n'/\sqrt{K}) - g(n/\sqrt{K}) = \frac{n'-n}{\sqrt{K}} g'(\xi).
\end{equation}
Now, $n/\sqrt{K}<n'/ \sqrt{K}<1/\sqrt{2}$  by \eqref{eq:n/sqrt(K)<=2}, and so $\xi\in \left(0,\frac {1}{\sqrt{2}}\right)$. In this range the derivative $$g'(u) = \frac{1-2u^2}{\sqrt{1-u^2}}$$ is positive and decreasing until it vanishes at $u=\frac{1}{\sqrt{2}}$. The upshot is
that so long as we stay away from this only zero, \eqref{eq:MVT} yields a bound of the desired type \eqref{eq:g diff >> sqrt(K)(n'-n)} (which is why we separately treated the case $n>\epsilon\sqrt{K}$ in the first place). To this end we further subdivide the interval $(0,1/\sqrt{2})$:
first, assuming $\frac{n}{\sqrt{K}}<\frac{n'}{\sqrt{K}}< \frac{1}{2}$ (allowed since $n/\sqrt{K}\le \epsilon$), \eqref{eq:MVT} reads
\begin{equation}
\label{eq:g(n'/sqrt(K)) increasae g'}
g(n'/\sqrt{K}) - g(n/\sqrt{K}) \ge \frac{1}{\sqrt{3}} \cdot \frac{n'-n}{\sqrt{K}},
\end{equation}
since for $\xi\in (0,1/2)$, one has $g'(\xi)\ge g'(1/2) = \frac{1}{\sqrt{3}}$ as $g'(\cdot)$ is decreasing. Otherwise, \eqref{eq:g(n'/sqrt(K)) increasae g'} holds true on the full range $\xi\in (0,1/\sqrt{2})$ with the constant $1/\sqrt{3}$ replaced by a slightly smaller constant
(but still bigger than the $\frac{1}{10}$ claimed in \eqref{eq:g diff >> sqrt(K)(n'-n)}), since $g$ is increasing.

 Inserting the nontrivial bound \eqref{eq:n'm'-nm nontriv bound} into  the r.h.s. of \eqref{eq:int not vanish=> >=1} we have:
\begin{equation*}
\begin{split}
&\left|\frac{1}{n^{2}}+\frac{1}{m^{2}} - \frac{1}{n'^{2}}-\frac{1}{m'^{2}}\right| \gg\frac{K^{3/2}(nm+n'm')(n'-n)}{n^{2}m^{2}n'^{2}m'^{2}} \\&\gg
\frac{(nm+n'm')(n'-n)}{n^{2}n'^{2} m} \gg \frac{(n'-n)}{n^{2}n'}.
\end{split}
\end{equation*}
However,
$$\frac{(n'-n)}{n^{2}n'}\gg \frac{1}{n^{3}},$$
since the ratio of the l.h.s. to the r.h.s. is
\[
\frac{(n'-n)/n^2 n'}{1/n^3} = \frac{n(n'-n)}{n+(n'-n)} = \frac{1}{\frac 1{n'-n}+\frac 1n} \geq \frac 12.
\]
This yields
\begin{equation}
\label{eq:diff>>sum n sm}
\begin{split}
&\left|\frac{1}{n^{2}}+\frac{1}{m^{2}} - \frac{1}{n'^{2}}-\frac{1}{m'^{2}}\right| \gg
\frac{1}{n^{3}} \ge \frac{1}{n^{4}} \\&\gg
\frac{1}{n^{4}}+
\frac{1}{m^{4}} + \frac{1}{n'^{4}} + \frac{1}{m'^{4}}.
\end{split}
\end{equation}

All in all, in either case \eqref{eq:diff>>sum eps} or \eqref{eq:diff>>sum n sm}
yield the second statement \eqref{eq:1/n^2+1/m^2-1/n'^2+1/m'^2 large} of Proposition \ref{prop:error majorized}.
The third statement \eqref{eq:Enm+En'm'<f4*diff} of Proposition \ref{prop:error majorized} follows directly
from \eqref{eq:1/n^2+1/m^2-1/n'^2+1/m'^2 large},
on recalling \eqref{eq:err term Gnm} and \eqref{eq:f2 def}.
\end{proof}


\begin{proposition}
\label{prop:no mult n=0}
There exists $\sigma_{0}>0$ so that for all $\sigma\in (0,\sigma_{0})$, if $n,n',m'\geq 1$ then $\eigen_{n,0}(\sigma)\neq \eigen_{n',m'}(\sigma)$.
\end{proposition}

\begin{proof}

It follows from Proposition \ref{prop:Gnm' expand} that
\begin{equation}
\label{eq:exp G around square}
\eigen_{n',m'}(\sigma) =  \pi^2(n'^2+m'^2)+8\sigma-\left( \frac{1}{n'^{2}}+\frac{1}{m'^{2}} \right)\cdot (\sigma^{2}+O(\sigma^{3}))
=  \pi^2(n'^2+m'^2)+8\sigma+O(\sigma^{2}),
\end{equation}
where the contribution of the error term $E_{(n',m')}$ is absorbed inside the $O(\sigma^{3})$, and
\begin{equation}
\label{eq:exp G0m square}
\eigen_{n,0} = \pi^2n^2  +6\sigma+ O(\sigma^{2}),
\end{equation}
with the constant involved in the $`O'$-notation in both \eqref{eq:exp G around square} and \eqref{eq:exp G0m square} absolute.
Hence, for $\sigma>0$ sufficiently small, if $\eigen_{n,0}(\sigma) = \eigen_{n',m'}(\sigma)$ then necessarily  $n'^{2}+m'^{2}=m^2$.

\vspace{2mm}

Next, if $n'^{2}+m'^{2}=n^2$ then from \eqref{eq:exp G around square} and \eqref{eq:exp G0m square} we obtain
\begin{equation*}
\label{eq:exp axis dominated}
\eigen_{n,0}(\sigma) = k_{0}(\sigma)^{2}+k_{m}(\sigma)^{2}<k_{n'}(\sigma)^{2}+k_{m'}(\sigma)^{2}=\eigen_{n',m'}(\sigma).
\end{equation*}
It follows trivially that for all $\sigma>0$, $(n,m)\ne (0,0)$, one has
$\eigen_{0,0}(\sigma)<\eigen_{n,m}(\sigma)$.
\end{proof}

\begin{proof}[Proof of Theorem \ref{thm:no multiplicities sqr}]

The statement of Theorem \ref{thm:no multiplicities sqr} is equivalent to having no relations $$\eigen_{n,m}(\sigma)=\eigen_{n',m'}(\sigma),$$
for $\sigma$ sufficiently small, $(n,m)\ne (n',m')$, where, once again, we assume w.l.o.g. that $n\le m$, $n'\le m'$ (recall that for
$L=1$, $\eigen_{n,m}(\cdot)=\eigen_{m,n}(\cdot)$). By Proposition
\ref{prop:no mult n=0}, we may further assume that $n,m,n',m'\ge 1$,
so use Proposition \ref{prop:Gnm' expand} to write
\begin{equation}
\label{eq:Gnm' expand err}
\eigen_{n,m}'(\sigma)=4\left(2+\frac{f_{2}(\sigma)}{\pi^{2}}\left(\frac{1}{n^{2}}+\frac{1}{m^{2}} \right)\right)+E_{(n,m)}(\sigma),
\end{equation}
with error term given by \eqref{eq:err term Gnm}.
Using Lemma \ref{lem:deriv eval}(b) we compute 
\begin{equation}\label{eq:Gnm''(0) eval}
\eigen_{n,m}''(0) = -8\left(\frac{1}{(\pi n)^{2}} +  \frac{1}{(\pi m)^{2}}  \right).
\end{equation}  
Writing the analogue of \eqref{eq:Gnm''(0) eval} for $(n',m')$ in place of $(n,m)$, and together with Proposition~\ref{prop:error majorized}(a), we deduce
that for $(n,m)$ and $(n',m')$ with $n^{2}+m^{2}=n'^{2}+m'^{2}$
and $n'm'>nm$, there exists some (a priori dependent on $(n,m)$ and $(n',m')$) neighbourhood of the origin so that
\begin{equation*}
\eigen_{n',m'}(\sigma)>\eigen_{n,m}(\sigma).
\end{equation*}

To make this neighbourhood absolute, we compare the expansions \eqref{eq:Gnm' expand err} of $\eigen_{n,m}'(\cdot)$
for $(n,m)$ and $(n',m')$ with $n^{2}+m^{2}=n'^{2}+m'^{2}$. We have
$\eigen_{n,m}'(0)= \eigen_{n',m'}'(0)$, and Proposition \ref{prop:error majorized}(a) and (c), bearing in mind that $f_{2}(\sigma)<0$
for all $\sigma>0$, implies that there exists some absolute
$\sigma_{0}>0$ so that
\begin{equation*}
\eigen_{n,m}'(\sigma)< \eigen_{n',m'}'(\sigma)
\end{equation*}
on $\sigma\in (0,\sigma_{0}]$, which concludes the proof of Theorem \ref{thm:no multiplicities sqr} for $(n,m)$ and $(n',m')$ on the same circle.

Finally, if $(n,m)$ and $(n',m')$ are not on the same circle, then \eqref{eq:Gnm asymp nm} shows that
$\eigen_{n,m}'(\cdot)-\eigen_{n',m'}'(\cdot)$ is bounded by an absolute constant around the origin (any bound $B>0$ could be taken for
sufficiently small neighbourhood of the origin).
Therefore, since $$\left|\eigen_{n,m}(\cdot)-\eigen_{n',m'}(\cdot)\right| \ge 1,$$ for $\sigma>0$ sufficiently small,
$\eigen_{n,m}(\cdot)-\eigen_{n',m'}(\cdot)$ maintains its sign.
\end{proof}

\subsection{Existence of spectral degeneracies} 
\begin{proposition}
\label{prop:mult (3,4),(1,5)}
There exist a number $\sigma>0$ so that $\eigen_{3,4}(\sigma)=\eigen_{1,5}(\sigma)$.
\end{proposition}
\begin{proof}

By Lemma \ref{lem:kn sec expand}(a), and recalling the notation \eqref{eq:Gn,m as g-}, we have
$\eigen_{3,4}(0) =  25\pi^{2}$ and $\eigen_{1,5}(0) = 26\pi^{2}$,
whereas
$\eigen_{3,4}(+\infty) = 41\pi^{2} $ and $\eigen_{1,5}(+\infty) = 40\pi^{2}$.
Therefore, the continuous function $\sigma\mapsto\eigen_{3,4}(\sigma) - \eigen_{1,5}(\sigma)$ changes  sign, and
so, by the Intermediate Value theorem, it  vanishes at some $\sigma>0$, i.e. $\eigen_{3,4}(\sigma)= \eigen_{1,5}(\sigma)$,
as claimed.
\end{proof}

\section{Spectral degeneracies for rectangles: proof of theorems \ref{thm:no multiplicities rectangle}-\ref{thm:bnded deg bad approx}}

\subsection{Existence of multiplicities for irrational $L^{2}$}

The following theorem asserts that, on recalling the notation \eqref{eq:lambda_mnL def},
there exist relations of the type
\begin{equation}
\label{eq:spec degen only one 0}
\eigen_{L;n,m}(\sigma) = \eigen_{L;0,m'}(\sigma),
\end{equation}
with $\sigma>0$ arbitrarily small,
and $n,m,m'\ge 1$ (depending on $\sigma$). This in particular implies Theorem \ref{thm:no multiplicities rectangle}.
For some $L>0$, spectral degeneracies of the type \eqref{eq:spec degen only one 0} subject to $n,m,m'\ge 1$ are the only degeneracies, at least for $\sigma>0$ sufficiently small, see Theorem \ref{thm:L diopant no mult} below.

\begin{theorem}
\label{thm:mult axes}
Let $ L^{2} $ be a positive irrational number. Then there exists a sequence of Robin parameters $\sigma_{j}\searrow 0$ and triples of positive integers $n,m,m'\geq 1 $ (depending on $\sigma_j$), so that
\begin{equation}
\label{eq:Gnm=G0m'}
\eigen_{L;n,m}(\sigma) = \eigen_{L;0,m'}(\sigma).
\end{equation}
\end{theorem}

The following result will be required towards giving a proof of Theorem \ref{thm:mult axes}.

\begin{lemma}\label{lem:Oppenheim}
Let $\theta>0$ be a positive irrational number. For every $  \epsilon>0$, there are positive integer solutions $n,m,m'>0$ of the inequality
\begin{equation}
\label{eq:Q(n,m,m') in (-eps,0)}
-\epsilon< n^{2}\theta+m^{2}-m'^{2}  <0  .
\end{equation}
\end{lemma}
\begin{proof}
For any irrational $\theta$, Hardy and Littlewood \cite{HL} proved in 1914 that the sequence of fractional parts $\{\theta n^2 \bmod 1:n=1,2,\dots\}$ is dense in the unit interval $[0,1)$ (improved to uniform distribution by Weyl shortly afterwards). Thus there are $n_1\gg 1$  and $j=j(n_1)\gg 1$ for which
$$
-\frac{\epsilon}{4} <\theta n_1^2 - j <0 .
$$
Multiplying by $4$ we obtain
$$
-\epsilon <\theta(2n_{1})^2 -4j <0 .
$$
Let
$$
n=2n_1,\quad m=j-1,\quad m'=j+1,,
$$
(which are positive). Then $m'^2-m^2 = 4j$, and we obtain
$$
-\epsilon <\theta n^2 +m^2-m'^2 <0
$$
with $n,m,m'>0$, as required.
\end{proof}

\begin{remark}
The proof of Lemma \ref{lem:Oppenheim} constructs infinitely many triples satisfying \eqref{eq:Q(n,m,m') in (-eps,0)}.
\end{remark}

\begin{proof}[Proof of Theorem \ref{thm:mult axes}]
Take any sequence $\epsilon_{j}\rightarrow 0$,
and find a triple of positive integers $(n,m,m')$ (depending on $\epsilon_j$) as in Lemma \eqref{lem:Oppenheim}, so that we have,
for $\theta=L^2$,
\begin{equation}
\label{eq:Gnm(0)-G0m'(0)>-epsk}
-\epsilon_{j} < \eigen_{L;n,m}(0) - \eigen_{L;0,m'}(0)= \frac{\pi^{2}}{\theta}\cdot \Big(  n^{2}\theta+m^{2}-m'^{2}  \Big)< 0.
\end{equation}
Next note that, for every $L>0$ and integers $n,m\geq 1$, one has
\begin{equation}
\label{eq:Galpha rect err}
\begin{split}
\eigen_{L;n,m}'(\sigma) = 4\left( \left(1+\frac{1}{L} \right) +\frac{1}{\pi^{2}}\left( \frac{f_{2}(\sigma)}{n^{2}} + \frac{f_{2}(L\cdot \sigma)}{L m^{2}} \right)  \right) + E_{(n,m)}(\sigma),
\end{split}
\end{equation}
where the error term is still bounded by \eqref{eq:err term Gnm}.
Indeed, in accordance with Lemma \ref{lem:deriv eval}, one has
\begin{equation*}
\left({k_{L;n}(\sigma)^{2}}\right)' = \frac{1}{L}\left({k_{n}(\cdot)^{2}}\right)'|_{L \sigma}
=\frac{4}{L}\left(1+\frac{f_{2}(L \sigma)}{\pi ^{2} n^{2}}\right)+O\left(\frac{\sigma^{4}}{n^{4}}\right),
\end{equation*}
and the rest follows from the definition of $\eigen_{L;n,m}(\cdot)$.

Now comparing \eqref{eq:Galpha rect err} to \eqref{eq:k0 asympt 2 terms},
we have the expansions: 
\begin{equation*}
\eigen_{L;n,m}(\sigma) = \eigen_{L;n,m}(0) +4\sigma\cdot \left(1+\frac{1}{L} \right) +O(\sigma^{2})
\end{equation*}
and
\begin{equation}
\label{eq:RNgap0,m' asympt}
\eigen_{L;0,m'}(\sigma) = \eigen_{L;0,m'}(0)+\sigma\cdot\left(2+\frac{4}{L}\right) + O(\sigma^{2}).
\end{equation}
Therefore, the difference between the Robin eigenvalues is given by
\begin{equation}
\label{eq:Gnm(R)-G0m'(R) exp}
\eigen_{L;n,m}(\sigma) - \eigen_{L;0,m'}(\sigma) = \eigen_{L;n,m}(0) - \eigen_{L;0,m'}(0) + 2\sigma + O(\sigma^{2}).
\end{equation}
In particular, if we choose $\sigma=\epsilon_{j}$, \eqref{eq:Gnm(0)-G0m'(0)>-epsk} with \eqref{eq:Gnm(R)-G0m'(R) exp} together imply that,
for $j$ sufficiently large,
\begin{equation*}
\eigen_{L;n,m}(\sigma) - \eigen_{L;0,m'}(\sigma)  > 0.
\end{equation*}
Therefore, by the Intermediate Value Theorem, there is some $\sigma_{j}\in (0,\epsilon_{j})$ so that
the equality \eqref{eq:Gnm=G0m'} holds, which is the claimed multiplicity.
\end{proof}

\subsection{A bound on multiplicities for badly approximable $L^{2}$}

Theorem \ref{thm:L diopant no mult}(a) asserts that if $\theta:=L^{2}$ is badly approximable in the sense of \eqref{eq:badly approx def},
then the only possible spectral degeneracies are either the type $\eigen_{L;n,m}(\sigma)=\eigen_{L;n',0}(\sigma)$ or
$\eigen_{L;n,m}(\sigma)=\eigen_{L;0,m'}(\sigma)$ for some $n,m,n',m'\ge 0$. Theorem \ref{thm:L diopant no mult}(b)-(c) will deduce
the bound for the spectral degeneracies claimed as part of Theorem \ref{thm:bnded deg bad approx}.

\begin{theorem}
\label{thm:L diopant no mult}
Asume that $L^2$ is badly approximable.

\begin{enumerate}[a.]

\item
For $\sigma_{0}>0$ sufficiently small, for
all $\sigma\in [0,\sigma_{0}]$ there are no spectral multiplicities $\eigen_{L;n,m}=\eigen_{L;n',m'}$ for $(n,m)\neq (n',m')$, with all $n,m,n',m'\geq 1$.

\item For $\sigma\in [0,\sigma_{0}]$ sufficiently small all multiplicities are bounded by $3$, i.e.
all eigenspaces are of dimension at most $3$.

\item If, in addition, $L$ is badly approximable, then all multiplicities
are bounded by $2$.

\end{enumerate}

\end{theorem}

\begin{proof}[Proof of Theorem \ref{thm:L diopant no mult}(a)]

We will show that, under the   hypotheses of Theorem \ref{thm:L diopant no mult}, the sign of
$$\eigen_{L;n,m}(\sigma) - \eigen_{L;n',m'}(\sigma),$$ that does not vanish at the origin, will be maintained in a neighborhood of the origin which is independent of $n,m,n',m'$. At this point we will assume for simplicity that 
$n\ne n'$. We will further assume w.l.o.g. that
$\eigen_{L;n,m}(0) > \eigen_{L;n',m'}(0)$. Abbreviating
$$\theta:=L^2,$$
then necessarily
\begin{equation}
\label{eq:G-G exp dioph}
\begin{split}
\eigen_{L;n,m}(0) - \eigen_{L;n',m'}(0) &= \pi \theta \left((n^{2}\cdot \theta + m^{2})- (n'^{2}\cdot \theta + m'^{2})  \right)
\\&= \pi \theta \left((n^{2}-n'^{2})\cdot \theta +(m^{2}-m'^{2})  \right) \gg (n^{2}-n'^{2})^{-1},
\end{split}
\end{equation}
since $\theta=L^2$ is badly approximable.
On the other hand, \eqref{eq:Galpha rect err} implies that
\begin{equation}
\label{eq:GL-GL exp}
\begin{split}
&\eigen_{L;n,m}(\sigma) - \eigen_{L;n',m'}(\sigma) \\&= \eigen_{L;n,m}(0) - \eigen_{L;n',m'}(0)+
O\left( \sigma^{2}\cdot \left(\left|\frac{1}{n^{2}}+ \frac{1}{n'^{2}}+ \frac{1}{m^{2}}+\frac{1}{m'^{2}} \right|  \right)\right),
\end{split}
\end{equation}
also following from a more direct argument, i.e. a truncated version of the expansion \eqref{eq:k'/2k asymp}, where the error term is of smaller
order of magnitude compared to the secondary term in \eqref{eq:k'/2k asymp}.
Note that the statement of Theorem \ref{thm:L diopant no mult}(a) is trivial, unless the l.h.s. of \eqref{eq:G-G exp dioph} is $< 1$, which will be assumed from now on,  so that
\begin{equation}
\label{eq:n2-n'2<<m2-m'2}
|n^{2}-n'^{2}| \ll |m^{2}-m'^{2}|.
\end{equation}

We claim that
\begin{equation*}
\frac{1}{|n^{2}-n'^{2}|} \gg \left|\frac{1}{n^{2}}+ \frac{1}{n'^{2}}+ \frac{1}{m^{2}}+\frac{1}{m'^{2}}\right|,
\end{equation*}
so that, bearing in mind \eqref{eq:G-G exp dioph}, the main term on the r.h.s. of \eqref{eq:GL-GL exp} is at least of the same order
of magnitude as the error term on the r.h.s. of \eqref{eq:GL-GL exp} not factoring in the factor $\sigma^{2}$, implying
no multiplicities for $\sigma$ sufficiently small.
First, clearly,
\begin{equation*}
\frac{1}{|n^{2}-n'^{2}|} \gg \frac{1}{n^{2}} + \frac{1}{n'^{2}},
\end{equation*}
so we are left to deal with bounding
\begin{equation}
\label{eq:1/n2-n'2>>1/m2-m'2}
\frac{1}{|n^{2}-n'^{2}|} \gg \frac{1}{m^{2}} + \frac{1}{m'^{2}}.
\end{equation}
To this end, we use \eqref{eq:n2-n'2<<m2-m'2} to obtain
\begin{equation*}
\frac{1}{|n^{2}-n'^{2}|} \gg \frac{1}{|m^{2}-m'^{2}|}\gg \frac{1}{m^{2}} + \frac{1}{m'^{2}},
\end{equation*}
which is \eqref{eq:1/n2-n'2>>1/m2-m'2}.
\end{proof}

\begin{proof}[Proof of Theorem \ref{thm:L diopant no mult}(b)-(c)]

Thanks to Theorem \ref{thm:L diopant no mult}(a), for $\sigma$ sufficiently small, a spectral multiplicity is either of the type
\begin{equation}
\label{eq:mult y axis}
k_{n}(\sigma)^{2}+\frac{1}{\theta}\cdot k_{m}(\sigma\cdot L)^{2} = k_{0}(\sigma)^{2}+\frac{1}{\theta}\cdot k_{m'}(\sigma\cdot L)^{2}
\end{equation}
or
\begin{equation}
\label{eq:mult x axis}
k_{n}(\sigma)^{2}+\frac{1}{\theta}\cdot k_{m}(\sigma\cdot L)^{2} = k_{n'}(\sigma)^{2}+\frac{1}{\theta}\cdot k_{0}(\sigma\cdot L)^{2},
\end{equation}
or
\begin{equation}
\label{eq:mult x axis y axis}
k_{n}(\sigma)^{2}+\frac{1}{\theta}\cdot k_{0}(\sigma\cdot L)^{2} = k_{0}(\sigma)^{2}+\frac{1}{\theta}\cdot k_{m'}(\sigma\cdot L)^{2},
\end{equation}
for some $n,m,n',m' \in \Z_{\ge 1}$ (taking into account that $k_{0}(\sigma)^{2}+\frac{1}{\theta}\cdot k_{0}(\sigma\cdot L)^{2}$
is arbitrarily small for $\sigma$ sufficiently small). Given $(n,m)\in \Z_{\ge 1}^2$ or $m'\in\Z_{\ge 1}$,   at most one lattice point $(n',0)$ can possibly satisfy either \eqref{eq:mult x axis} or
\eqref{eq:mult x axis y axis} (resp. \eqref{eq:mult y axis} or \eqref{eq:mult x axis y axis}), and the same holds analogously for $(0,m')$.
It follows that the multiplicities are bounded by $3$, concluding Theorem \ref{thm:L diopant no mult}(b).

The above also shows that if multiplicity $3$ actually occurs with $\sigma$ arbitrarily small, then
\begin{equation}
\label{eq:triple equality}
k_{n}(\sigma)^{2}+\frac{1}{\theta}\cdot k_{m}(\sigma\cdot L)^{2} = k_{0}(\sigma)^{2}+\frac{1}{\theta}\cdot k_{m'}(\sigma\cdot L)^{2} = k_{n'}(\sigma)^{2}+\frac{1}{\theta}\cdot k_{0}(\sigma\cdot L)^{2}
\end{equation}
is satisfied. Using only the latter of the two equalities of \eqref{eq:triple equality}, together with \eqref{eq:k'/2k asymp} and \eqref{eq:k0 asympt 2 terms}
we obtain:
\begin{equation}
\label{eq:quad form 2var small}
0 = \left(k_{0}(\sigma)^{2}+\frac{1}{\theta}\cdot k_{m'}(\sigma\cdot L)^{2}\right)-\left(k_{n'}(\sigma)^{2}+\frac{1}{\theta}\cdot k_{0}(\sigma\cdot L)^{2}\right) = \frac{\pi^{2}}{\theta}(m'^{2}-n'^{2}\theta)+O(\sigma).
\end{equation}
In particular, if, as we assumed, \eqref{eq:triple equality} occurs with $\sigma$ arbitrarily small (i.e. \eqref{eq:quad form 2var small} holds for a sequence $\sigma_{j}\rightarrow 0$),
then the quadratic form
\begin{equation}
\label{eq:quad form n'^2theta-m'^2}
n'^{2}\theta - m'^{2}
\end{equation}
attains arbitrarily small values. However,
\begin{equation*}
n'^{2}\theta - m'^{2} = (n'L-m')\cdot (n'L+m') \gg \frac{1}{n'}\cdot (n'L+m') \gg 1
\end{equation*}
by the assumption on $L$ being badly approximable, contradicting our conclusion on the quadratic
form \eqref{eq:quad form n'^2theta-m'^2} attaining arbitrarily small values.
\end{proof}

\subsection{A bound on the number of degenerate eigenvalues}

Recall Weyl's law \eqref{eq:N spec func def} for the Robin spectrum of $\Rc_{L}$, and the function $N^{\rm mult}(\lambda) $ counting
the number of multiple eigenvalues $\le \lambda$, including multiplicities.
As a corollary of the arguments above, we deduce the bound \eqref{eq:Nmult << sqrt(lambda)} for $N^{\rm mult}(\lambda) $,
thus concluding the proof of Theorem \ref{thm:bnded deg bad approx}.

 \begin{corollary}
 If $L^2$ is badly approximable, then
 \[
 N^{\rm mult}(\lambda) \ll \sqrt{\lambda}
 \]
 \end{corollary}
 \begin{proof}
Indeed, we saw that the only possible source of multiplicities is when $\eigen_{L;n,m} = \eigen_{L;n',0}$ or $\eigen_{L;n,m} = \eigen_{L;0,m'}$, and that each source, e.g. $ \eigen_{L;0,m'}$, coming from one of the axes can at most contribute a three-fold degeneracy, because we cannot have
$ \eigen_{L;0,m'}= \eigen_{L;0,m''}$ with $m'\neq m''$. Moreover, the number of eigenvalues
$$\eigen_{L;0,m'}=k_0(\sigma)^2+\left(\frac 1L k_{m'}(L\sigma)\right)^2\leq \lambda
$$
is at most the number of $m'\geq 0$ with $k_{m'}(L\sigma)\leq  L\sqrt{\lambda}$, which is $O(\sqrt{\lambda})$ since $k_{m'} =n\pi+O(1)$.
\end{proof}

\subsection{An auxiliary result}

For future reference we record the following result concerning the asymptotic behaviour of the RN gaps, similar to one used previously, but simpler in that it has no control over the error term as $\sigma>0$ is varying:
\begin{proposition}
\label{prop:gap constant ae}
For $n,m\geq 0$,
\[
 \Lambda_{L;n,m}(\sigma) - \Lambda_{L;n,m}(0) = \left(1+\frac 1{L }\right)4\sigma +O_{L,\sigma}\left(\frac 1{1+n^2}+\frac 1{1+m^2}\right) .
\]
\end{proposition}

\begin{figure}[ht]
\begin{center}
\includegraphics[height=80mm]{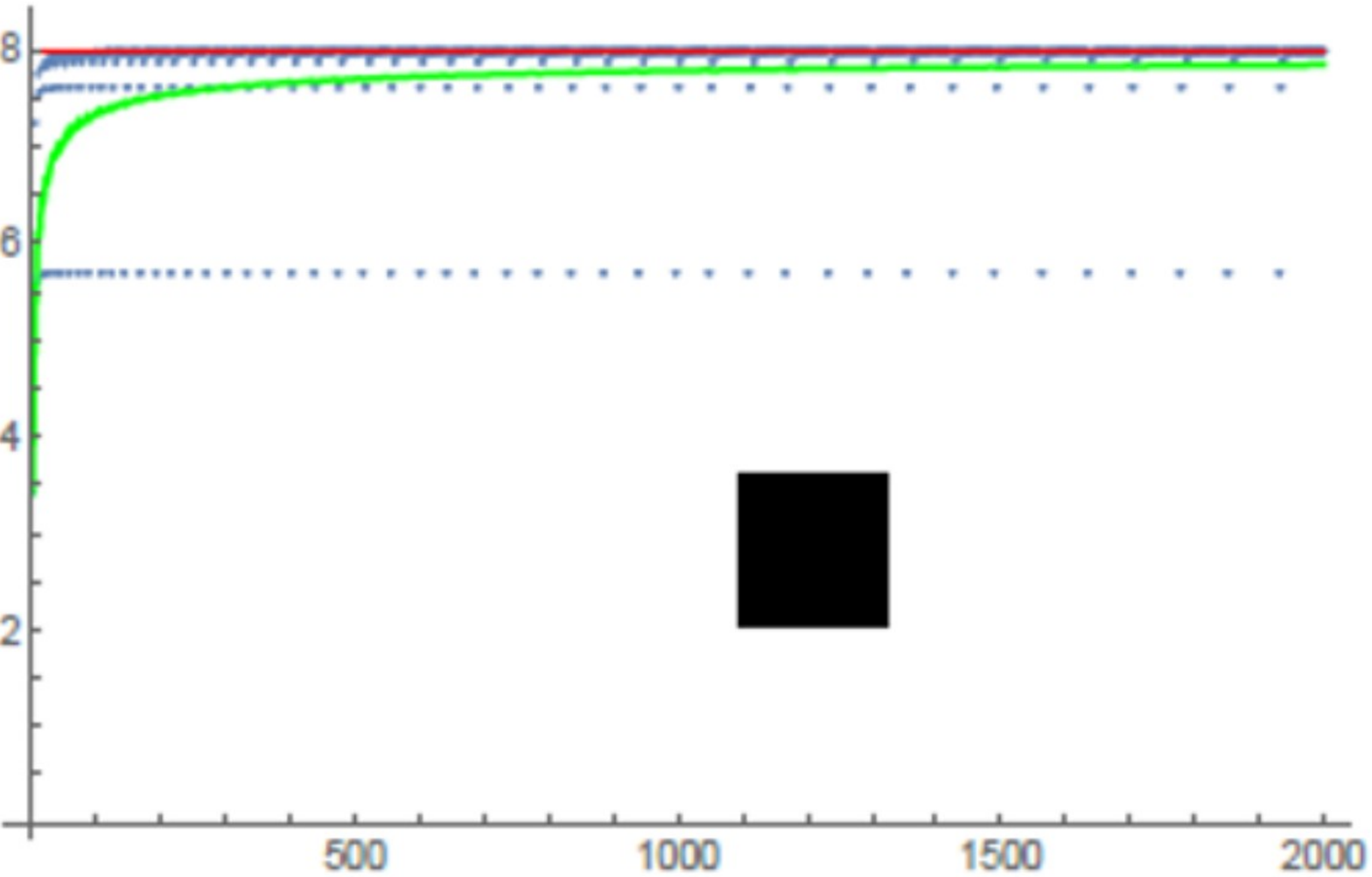}
\caption{$2000$ RN gaps for the square, $\sigma=1$. The bulk of the RN gaps tend to the mean $8$. The secondary
curves correspond to lattice points whose minimal coordinate is small, in particular, lattice points lying on the axes, whose
$RN$ gaps are less than $6$. Red: the mean $8$, green: moving average.}
\label{fig:Rn gaps square nums}
\end{center}
\end{figure}

Note that the ``main term'' is
\begin{equation}
\label{eq:main term RWY}
\left(1+\frac 1{L }\right)4\sigma  = 2\frac{\length \partial \Rc_L}{\area \Rc_L} \sigma
\end{equation}
which is the mean value of the RN gaps, by our general theory \cite{RWY}.
 For $n$ fixed, $m\rightarrow\infty$, the corresponding sequence
of RN gaps is
\begin{equation*}
\begin{split}
&\Lambda_{L;n,m}(\sigma) - \Lambda_{L;n,m}(0) = k_{n}(\sigma)^{2}-k_{n}(0)^{2} +
\frac{1}{L^{2}}\left( k_{m}(\sigma L)^{2}-k_{m}(0)^{2}\right) \rightarrow \\&k_{n}(\sigma)^{2}-k_{n}(0)^{2} + \frac{1}{L^{2}} 4 \sigma L
=k_{n}(\sigma)^{2}-k_{n}(0)^{2} + 4 \frac{\sigma}{L},
\end{split}
\end{equation*}
by Lemma \ref{lem:deriv eval}(d), and we observe that
\begin{equation*}
k_{n}(\sigma)^{2}-k_{n}(0)^{2} + 4 \frac{\sigma}{ L} < 4\sigma(1+1/L)
\end{equation*}
by Lemma \ref{lem:deriv eval}(b), at least, for $\sigma$ sufficiently small.
That is, the RN gaps for this infinite (though rare) sequence of energies
are strictly less than the mean \eqref{eq:main term RWY}. In particular, for $n=0$, we obtain, recalling Lemma
\ref{lem:kn sec expand}(b),
$$k_{0}(\sigma)^{2}-k_{n}(0)^{2} + 4 \frac{\sigma}{L}< 2\sigma +4 \frac{\sigma}{L}= 4\sigma(1/2+1/L) ,$$
again, at least for $\sigma$ sufficiently small. Likewise, one may obtain infinite sequences of RN gaps
that are asymptotic to a value strictly less than \eqref{eq:main term RWY} by fixing $m$ and taking $n\rightarrow\infty$.

Figure \ref{fig:Rn gaps square nums} illustrates $2000$ RN gaps for the square, $\sigma=1$. Here
$k_{0}(\sigma)^{2}-k_{n}(0)^{2} + 4=5.707\ldots$ corresponding to the bottom trend line in the picture,
and $k_{1}(\sigma)^{2}-k_{n}(1)^{2} + 4=7.62275\ldots$ corresponding to the second to bottom trend line, etc.

\begin{proof}

The statement of Proposition \ref{prop:gap constant ae} follows directly from \eqref{eq:Galpha rect err} for $n,m\ge 1$ and
from \eqref{eq:RNgap0,m' asympt} for $n=0$, $m\ge 1$
(and the trivial bound $\Lambda_{L;0,0} = O(1)$).
\end{proof}

\section{Boundedness of Robin-Neumann gaps: Proof of Theorem \ref{thm:RN gaps bnded}}

\begin{lemma}
\label{lem:kn^2<<sigma}
There exists an absolute constant $C_{0}>0$, so that for all $n\ge 0$ and $\sigma>0$,
\begin{equation}
\label{eq:kn(sigma)-kn(0)<C0sigma}
k_{n}(\sigma)^{2}-k_{n}(0)^{2} \le C_{0}\cdot \sigma.
\end{equation}

\end{lemma}

\begin{proof}

For $\sigma>c_{0}\cdot (n+1)$ with $c_{0}$ sufficiently small parameter to be chosen later, we use the trivial bound
\begin{equation*}
k_{n}(\sigma)^{2}-k_{n}(0)^{2} = (k_{n}(\sigma)-k_{n}(0))\cdot (k_{n}(\sigma)+k_{n}(0)) \le \pi \cdot 2(n+1)\pi \le \frac{2\pi^{2}}{c_{0}}\cdot \sigma.
\end{equation*}
Otherwise, for $\sigma\le c_{0}\cdot (n+1)$ with $c_{0}$ sufficiently small, assume that $n\ge 1$, and will take care of $n=0$ separately below. Recall the secular equation
\begin{equation}
\label{eq:secular gen R b}
\tan(k) = \frac{2\sigma k}{k^{2}-\sigma^{2}}.
\end{equation}
In this case, the denominator on the r.h.s. of \eqref{eq:secular gen R b} is bounded away from $0$,
so that the r.h.s. of \eqref{eq:secular gen R b} is
$\le 4\frac{\sigma}{k}<4c_{0}\frac{n+1}{n\pi}$ arbitrarily small by appropriately choosing $c_{0}$, and then, since $\arctan(x)\le x$ for $x>0$,
\begin{equation*}
k_{n}(\sigma)-k_{n}(0) = k_{n}(\sigma)-n\pi \le 4\frac{\sigma}{k_{n}(\sigma)}.
\end{equation*}
We then have
\begin{equation*}
\begin{split}
k_{n}(\sigma)^{2}-k_{n}(0)^{2} &= (k_{n}(\sigma)-k_{n}(0))\cdot (k_{n}(\sigma)+k_{n}(0)) \le 4\frac{\sigma}{k_{n}(\sigma)} \cdot 2(n+1)\pi
\le 8\pi \frac{n+1}{n\pi}\cdot\sigma \\&= 8\frac{n+1}{n} \cdot \sigma \le 16\sigma.
\end{split}
\end{equation*}

Finally, we take care of the remaining case of $n=0$, under the assumption $\sigma \le c_{0}(n+1) = c_{0}$. Recall that (Lemma
\ref{lem:kn sec expand}(b)) here $k_{0}(\sigma)>\sigma$. Denote $k=k_{0}(\sigma)< \frac{\pi}{2}$ for $\sigma<\frac{\pi}{2}$
(Proposition \ref{lem:kn sec expand}(c)). Therefore, we can use $k<\frac{\pi}{2}$, so that $k<\tan(k)$, and \eqref{eq:secular gen R} reads
\begin{equation*}
k<\tan{k} = \frac{2\sigma k}{k^{2}-\sigma^{2}},
\end{equation*}
and manipulate with that to write (recall that the denominator is positive)
\begin{equation*}
k^{2}-\sigma^{2} < 2\sigma,
\end{equation*}
and then
\begin{equation*}
k^{2}<2\sigma+\sigma^{2}<3\sigma,
\end{equation*}
valid for $\sigma < c_{0}$, provided that $c_{0}$ is sufficiently small.
\end{proof}

\begin{lemma}
\label{lem:kn^2>>sigma}

There exists an absolute constant $c_{0}>0$, so that for all $n\ge 0$ and $\sigma\in [0,1]$,
\begin{equation}
\label{eq:kn2>>sigma}
k_{n}(\sigma)^{2}-k_{n}(0)^{2} \ge c_{0}\cdot \sigma.
\end{equation}

\end{lemma}

\begin{proof}

The main argument behind the proof of Lemma \ref{lem:kn^2>>sigma} is similar to that of Lemma \ref{lem:kn^2<<sigma}.
First, assume that $n\ge 1$, so that here $k_{n}(\sigma)\ge \pi$, and the denominator of the r.h.s. of \eqref{eq:secular gen R}
is bounded away from $0$. It then follows that $\tan k > \frac{2\sigma}{k}$, and then
$$k_{n}(\sigma)-k_{n}(0) =k_{n}(\sigma)-n\pi > c_{1}\frac{\sigma}{k_{n}(\sigma)}$$
with $c_{1}$ sufficiently small. We then have for $n\ge 1$,
\begin{equation*}
k_{n}(\sigma)^{2}-k_{n}(0)^{2} = (k_{n}(\sigma)-k_{n}(0))\cdot (k_{n}(\sigma)+k_{n}(0)) > c_{1}\frac{\sigma}{k_{n}(\sigma)} \cdot k_{n}(\sigma)
= c_{1}\cdot\sigma
\end{equation*}
which is \eqref{eq:kn2>>sigma} with $c_{1}>0$ in place of $c_{0}$.
That \eqref{eq:kn2>>sigma} holds with $n=0$ follows directly from \eqref{eq:k0 asympt 2 terms}
for all $\sigma\in [0,1]$, at the expense of further decreasing the constant to some $c_{0}$.
\end{proof}

\begin{proof}[Proof of Theorem~\ref{thm:RN gaps bnded}]

Recall that the energies $\{\lambda_{j}^{L}(\sigma)\}_{j\ge 1}$ are the sorted list of $\{\eigen_{L;n,m}(\sigma)\}_{n,m\ge 0}$.
The main obstacle in inferring the upper and the lower bounds \eqref{eq:dn<Csig} and \eqref{eq:dn>csig}
of Theorem \ref{thm:RN gaps bnded} directly from the corresponding bounds
in lemmas \ref{lem:kn^2<<sigma} and \ref{lem:kn^2>>sigma}
is that the numbers $\eigen_{L;n,m}(\sigma)$ can mix, so that the gaps $d_{j}$ will not, in general, be equal
to $\eigen_{L;n,m}(\sigma)-\eigen_{L;n,m}(0)$. We will overcome this obstacle by appealing to an argument inspired by an idea behind
the proof of ~\cite[Theorem 1.7]{RWY}, for both \eqref{eq:dn<Csig} and \eqref{eq:dn>csig}. Recall the spectral function $N(\lambda)=N_{L;\sigma}(\lambda)$ as in \eqref{eq:N spec func def}. Set
\begin{equation}
\label{eq:AL def dep on L}
a_{L}:=1+\frac{1}{L} = \frac{1}{2} \frac{\length(\Rc_{L})}{\area(\Rc_{L})}.
\end{equation}

\vspace{2mm}

First we prove \eqref{eq:dn<Csig}.
Lemma~\ref{lem:kn^2<<sigma} yields a number $C_{0}>0$ so that if $t:=\lambda_{j}(0)$, then
for every $\sigma>0$ and $n,m\ge 0$ so that $\eigen_{L;n,m}(0) = t$, one has
$$\eigen_{L;n,m}(\sigma)  \le t+C_{0}\cdot \sigma+C_{0}\cdot \frac{1}{L^{2}}\sigma L = t+C_{0}\left(1+\frac{1}{L}\right)\sigma.
$$
We deduce that $N_{\sigma}(t+C_{0}a_{L}\cdot \sigma) \ge j$ with $a_{L}$ as in \eqref{eq:AL def dep on L},
and hence $\lambda_{j}(\sigma) \le t+C_{0}a_{L}\cdot\sigma$.
Finally, we infer $d_{j}(\sigma) = \lambda_{j}(\sigma)-t \le C_{0}a_{L}\cdot \sigma$, which, thanks to \eqref{eq:AL def dep on L} and \eqref{eq:d(sigma) def}, is identified as \eqref{eq:dn<Csig}.

Next, we show \eqref{eq:dn>csig}. Using the same idea as above,
Lemma \ref{lem:kn^2>>sigma} gives an absolute $c_{0}>0$ so that if $t:=\lambda_{j}(0)$, then
for every $\sigma\in [0,1]$ and $n,m\ge 0$ with $\eigen_{L;n,m}(0) = t$, one has $$\eigen_{L;n,m}(\sigma) \ge t+
c_{0}a_{L}\cdot \sigma.$$
Therefore, for every $t'<t+c_{0}\cdot \sigma$, $N_{\sigma}(t') < j$, and thus $$\lambda_{j}(\sigma)\ge t+c_{0}a_{L}\cdot\sigma.$$ Finally we obtain
 $$d_{j}(\sigma) = \lambda_{j}(\sigma)-t \ge c_{0}a_{L}\sigma$$
 which is \eqref{eq:dn>csig}, on recalling \eqref{eq:AL def dep on L} and \eqref{eq:d(sigma) def} again.
\end{proof}

\section{Pair correlation: Proof of Theorem \ref{thm:PC diophantine}}

Fix $f\in C_c(\R)$ even. The associated pair correlation function is
\[
R_2^\sigma(f,N):=\frac 1N \sum_{1 \leq j\neq k\leq N} f\left(\frac{\lambda_j(\sigma)-\lambda_k(\sigma)}{\bar s} \right),
\]
where the mean spacing $\overline{s}$ is given by \eqref{eq:sbar mean spac}.

\begin{proposition}\label{prop:comparison}
For any rectangle $\Rc_L$, and any fixed  $\sigma>0$
\[
\Big|  R_2^\sigma(f,N)-R_2^0(f,N) \Big| \ll N^{-1/10} \to 0.
\]
\end{proposition}
\begin{proof}
For notational convenience, in what follows we will neglect the asymptotically constant mean spacing \eqref{eq:sbar mean spac} being equal to $\frac{4\pi}{\area \Rc_L}$, and proceed as if $\{\lambda_{j}\}$ had mean spacing asymptotic to unity.
Note that a feature of the pair correlation function is that, by its definition, the ordering of the eigenvalues is irrelevant. Therefore we can compute it by taking, for some large $N\gg 1$,
\[
\widetilde N(\sigma)= \#\{k: \lambda_k(\sigma)\leq N\} = \#\{n,m\geq 0: \Lambda_{L;n,m}(\sigma)\leq N\}
\]
which, by Weyl's law, is asymptotically $\widetilde N(\sigma) \approx N$, and then
\[
 R_2^\sigma(f,\widetilde N) = \frac 1{\widetilde N(\sigma)}\sum_{ \substack{\Lambda_{L;n,m}(\sigma), \Lambda_{L;n',m'}(\sigma)\leq N\\
 (n,m)\neq (n',m')}}f\Big(\Lambda_{L;n,m}\left(\sigma ) -\Lambda_{L;n',m'} (\sigma\right) \Big).
\]
Since $\Lambda_{L;n,m}(\sigma) = \Lambda_{L;n,m}(0) + O_{\sigma}(1)$, we also know that $\widetilde N(\sigma) \sim \widetilde N(0)$.

Therefore we can bound the difference between the Neumann and Robin pair correlations as
\begin{equation}
\label{eq:pair corr from N to R}
\begin{split}
&|R_2^\sigma(f,\widetilde N) - R_2^0(f,\widetilde N) | \ll \\&\frac 1{N} \sum_{\substack{\Lambda_{L;n,m}(0), \Lambda_{L;n',m'}(0)  \leq N\\ (n,m)
\neq (n',m')}} \left|
f\left(\Lambda_{L;n,m}\left(\sigma ) -\Lambda_{L;n',m'} (\sigma\right) \right)-f\left(\Lambda_{L;n,m}(0) -\Lambda_{L;n',m'}(0) \right)
\right|
\end{split}
\end{equation}

Set
\[
d_{n,m}(\sigma):=\Lambda_{L;n,m}(\sigma) - \Lambda_{L;n,m}(0)
\]
(not to be confused with the actual RN gaps $\lambda_k(\sigma)-\lambda_k(0)$). These are bounded, say $d_{n,m}(\sigma)\leq C$ (which depends on $L$ and $\sigma$), moreover by Proposition~\ref{prop:gap constant ae}, there is $C_1>0$ so that
\begin{equation}\label{def of C_1}
\begin{split}
\left|d_{n,m}(\sigma)- \left(1+\frac 1L\right)4\sigma\right|&=\left |\Lambda_{L;n,m}(\sigma)-\Lambda_{L;n,m}(0) - \left(1+\frac 1L\right)4\sigma \right| \\&\leq C_1\left(\frac 1{1+n^2}+\frac 1{1+m^2}\right).
\end{split}
\end{equation}

Assume that $f$ is supported in $[-\rho,\rho]$.
Take a function $g\in C_c^\infty(\R)$ that is non-negative: $ g\geq 0$, and so that
\begin{equation}
\label{eq:g==maxf'}
g\equiv \max |f'| \quad {\rm on}  \quad  [-2(\rho+2C_1), 2(\rho+2C_1)],
\end{equation}
 where $C_1$ is as in \eqref{def of C_1}. In particular $g\geq |f'|$, as depicted in Figure~\ref{fig:dominatingfunction}.
\begin{figure}[ht]
\begin{center}
\includegraphics[height=40mm]{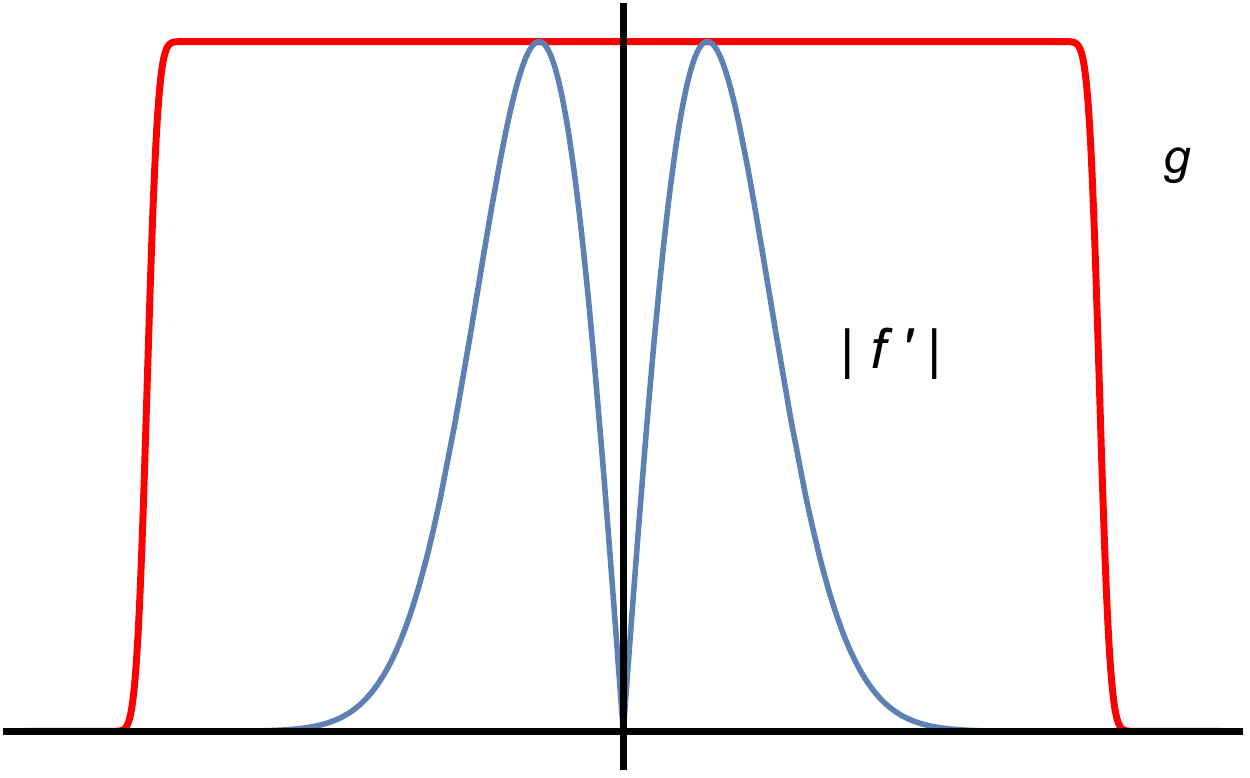}
\caption{ Sketch of $|f'|$ and $g$.}
\label{fig:dominatingfunction}
\end{center}
\end{figure}

We first show that
\begin{multline}\label{step 1 comparison}
\left|  R_2^\sigma(f,\widetilde{N})-R_2^0(f,\widetilde{N}) \right| \ll
\\
\frac 1N \sum  g\Big( \Lambda_{L;n,m}(0) - \Lambda_{L;n',m'}(0) \Big) \cdot \left(\frac 1{1+n^2}+\frac 1{1+m^2}
+\frac 1{1+n'^2} +\frac 1{1+m'^2}\right) .
\end{multline}
Indeed, to contribute to $R_2^\sigma(f,\widetilde{N})-R_2^0(f,\widetilde{N})$ in \eqref{eq:pair corr from N to R}, it is forced that at least one of the two eigenvalue differences
\begin{equation}
\label{eq:eig diff sigma, 0}
\Lambda_{L;n,m} (\sigma ) -\Lambda_{L;n',m'} (\sigma ), \; \Lambda_{L;n,m}(0) -\Lambda_{L;n',m'} (0)
\end{equation}
are in $ \supp f \subseteq [-\rho,\rho]$.
Since the difference between these two expressions \eqref{eq:eig diff sigma, 0} is $$d_{n,m}(\sigma)-d_{n',m'}(\sigma) \in [-2C_1,2C_1]$$ by \eqref{def of C_1}, if one of the expressions \eqref{eq:eig diff sigma, 0} is in $[-\rho,\rho]$, then both
\begin{equation}
\label{eq:eig diff sigma,0 bnd}
\Lambda_{L;n,m} (\sigma ) -\Lambda_{L;n',m'} (\sigma ), \, \Lambda_{L;n,m}(0) -\Lambda_{L;n',m'} (0) \in [-(\rho+2C_1),\rho+2C_1].
\end{equation}

For such a pair, we have by the mean value theorem
\begin{multline}\label{used mean value thm}
 f\Big(\Lambda_{L;n,m}\left(\sigma ) -\Lambda_{L;n',m'} (\sigma\right) \Big)-f\Big(\Lambda_{L;n,m}(0) -\Lambda_{L;n',m'}(0) \Big)
 \\
 = \Big(d_{n,m}(\sigma) - d_{n',m'}(\sigma)\Big) f'\Big(\xi(n,m,n',m')\Big)
\end{multline}
for some $\xi(n,m,n',m')$ between $\Lambda_{L;n,m}\left(\sigma ) -\Lambda_{L;n',m'} (\sigma\right) $ and $\Lambda_{L;n,m}(0) -\Lambda_{L;n',m'}(0)$. Proposition~\ref{prop:gap constant ae} implies that
\begin{equation}
\label{eq:d diff<=C1 sum}
 d_{n,m}(\sigma) - d_{n',m'}(\sigma) \leq C_1\left( \frac 1{1+n^2}+\frac 1{1+m^2} +\frac 1{1+n'^2}+\frac 1{1+m'^2}\right) .
\end{equation}
In addition, if some summand $$\left|
f\left(\Lambda_{L;n,m}\left(\sigma ) -\Lambda_{L;n',m'} (\sigma\right) \right)-f\left(\Lambda_{L;n,m}(0) -\Lambda_{L;n',m'}(0) \right)
\right|$$ on the r.h.s. of \eqref{eq:pair corr from N to R} does not vanish, then
\begin{equation}
\label{eq:f'(xi)<=g(diff)}
\left| f'\left(\xi(n,m,n',m')\right)\right|\leq \max|f'|  = g\left( \Lambda_{L;n,m}(0) -\Lambda_{L;n',m'}(0)  \right),
\end{equation}
by \eqref{eq:g==maxf'} and \eqref{eq:eig diff sigma,0 bnd}. The claimed inequality \eqref{step 1 comparison}
follows upon substituting \eqref{eq:f'(xi)<=g(diff)} and \eqref{eq:d diff<=C1 sum} into \eqref{used mean value thm}, and finally into
\eqref{eq:pair corr from N to R}.

\vspace{2mm}

Next, we claim that the r.h.s. of \eqref{step 1 comparison} satisfies the inequality
\begin{multline}\label{eq:next step}
\frac 1N \sum  g\left( \Lambda_{L;n,m}(0) - \Lambda_{L;n',m'}(0) \right) \cdot
\left(\frac 1{1+n^2}+\frac 1{1+m^2} +\frac 1{1+n'^2}+\frac 1{1+m'^2} \right)
 \ll N^{-\frac 1{10} }
\end{multline}
where the sum is over all pairs with $(n,m)\neq (n',m')$. 
To see this, we take a large parameter $M>0$ to be chosen later, and divide the summands into two categories: (1) those with $\min(n,m,n',m')>M$, and (2) the rest. An individual summand with $\min(n,m,n',m')>M$ is bounded
\[
 \frac 1{1+n^2}+\frac1{1+m^2} +\frac 1{1+n'^2}+\frac1{1+m'^2} \ll \frac 1{M^2},
\]
so that the total contribution of the summands of the 1st category is bounded by
\[
\frac 1N \sum  g\left( \Lambda_{L;n,m}(0) - \Lambda_{L;n',m'}(0) \right) \frac 1{M^2}  =  \frac 1{M^2} R_2^0(g,N),
\]
by forgetting the restriction $\min(n,m,n',m')>M$.
Since the pair correlation function for any rectangle is bounded \cite[Lemma 3.1]{BL} by
\[
 R_2^0(g,N) \ll_g N^\varepsilon
\]
it follows that the contribution to the sum \eqref{eq:next step} of summands of the $1$st category is dominated by
\begin{equation}
\label{eq:1st cat contr}
\ll \frac{N^\epsilon}{M^2} .
\end{equation}

We next treat the contribution to the sum \eqref{eq:next step} of $2$nd category summands, those with at least one of the coordinates small $\leq M$, say $n\leq M$, where use the trivial bound
\[
\frac 1{1+n^2}+  \frac 1{1+m^2} +\frac 1{1+n'^2}+\frac 1{1+m'^2}  \ll 1.
\]
Hence the contribution to the sum \eqref{eq:next step} of the $2$nd category summands is bounded by
\begin{multline}
\label{eq:2nd cat sum bnd}
\frac 1N \sum_{0\leq n\leq M}  \sum_{0\leq m \ll \sqrt{N}}
\sum_{0\leq n',m'\ll \sqrt{N}} g\left(  \Lambda_{L;n,m}(0) - \Lambda_{L;n',m'}(0)  \right)
\\
\ll \frac 1N  \sum_{0\leq n\leq M}  \sum_{0\leq m \ll \sqrt{N}}
  \#\left\{ (n',m')\in [0,\sqrt{N}]^2: | \Lambda_{L;n,m}(0) - \Lambda_{L;n',m'}(0)|\leq C_2 \right\},
\end{multline}
since $\supp g\subseteq [-C_2,C_2]$ with $C_{2}=\rho+2C_{1}$.
Given $(n,m)$, the term
$$\#\left\{ (n',m')\in [0,\sqrt{N}]^2: | \Lambda_{L;n,m}(0) - \Lambda_{L;n',m'}(0)|\leq C_2 \right\} $$
is the number of lattice points in a quarter of an elliptic annulus of width
$\ll \frac{1}{\sqrt{\Lambda_{L;n,m}(0)}}$ and constant area (both depending on $C_2$).

While we expect the number of points in such a narrow annulus to be very small, say $\ll N^\epsilon$, we are unable to show this for irrational $L^2$. Instead we give a crude bound of $\ll N^{1/3}$:
We use the classical bound on the number of lattice point in a dilated ellipse (for the circle this is due to Sierpinski in 1906)
\[
\#\{(n',m')\in \Z^2: \Lambda_{L;n',m'}(0)\leq x \} = Ax  + O(x^{1/3})
\]
where $A$ is the area of the ellipse.
Therefore a crude bound for the number of points in the annulus is
\begin{multline*}
\#\{ (n',m')\in\Z_{\geq 0}^2: | \Lambda_{L;n,m}(0) - \Lambda_{L;n',m'}(0)|\leq C_2\}   =
\\
 A( \Lambda_{L;n,m}(0) +C_2)   - A( \Lambda_{L;n,m}(0) -C_2)
 +O\Big( \Lambda_{L;n,m}(0) ^{1/3} \Big)
 \ll \Lambda_{L;n,m}(0)^{1/3} .
\end{multline*}
Since $\Lambda_{L;n,m}(0)\ll N$, we obtain
\begin{equation}
\label{eq:bnd lat pnts N^1/3}
\#\left\{ (n',m')\in \Z_{\geq 0}^2: | \Lambda_{L;n,m}(0) - \Lambda_{L;n',m'}(0)|\leq C_{2} \right\}  \ll N^{1/3} .
\end{equation}
Summing the inequality \eqref{eq:bnd lat pnts N^1/3}
(whose l.h.s. are clearly greater or equal than the summands on the r.h.s. of \eqref{eq:2nd cat sum bnd}) over $n\leq M$ and $m\ll \sqrt{N}$, and substituting into \eqref{eq:2nd cat sum bnd}
yields the bound
\begin{equation}
\label{eq:contr 2nd cat}
\begin{split}
&\frac 1N \sum_{0\leq n\leq M}  \sum_{0\leq m \ll \sqrt{N}}
\sum_{0\leq n',m'\ll \sqrt{N}} g\left(  \Lambda_{L;n,m}(0) - \Lambda_{L;n',m'}(0)  \right)
\\&\ll
\frac 1N \sum_{1\leq n\leq M}   \sum_{0\leq m \ll \sqrt{N}} N^{1/3} \ll  MN^{-1/6} .
\end{split}
\end{equation}
for the contribution of the $2$nd category summands.
Consolidating the contributions \eqref{eq:1st cat contr} and \eqref{eq:contr 2nd cat} of the $1$st
and the $2$nd categories summands respectively, we finally obtain a bound
for the sum in \eqref{eq:next step}:
\begin{equation*}
\begin{split}
&\frac 1N \sum  g\left( \Lambda_{L;n,m}(0) - \Lambda_{L;n',m'}(0) \right) \cdot
\left(\frac 1{1+n^2}+\frac 1{1+m^2} +\frac 1{1+n'^2}+\frac 1{1+m'^2} \right)
 \\&\ll  \frac{N^\epsilon}{M^2} +  MN^{-1/6} \ll N^{-1/10 }
 \end{split}
\end{equation*}
on taking $M=N^{1/18+\epsilon/3}$.
\end{proof}

\subsection*{Acknowledgements}

This research was supported by the European Research Council (ERC) under the European Union's Horizon 2020 research and innovation programme (Grant agreement No. 786758) and by the Israel Science Foundation (grant No. 1881/20). We wish to thank Nadav Yesha for many stimulating conversations,
 and a referee for a simplification of the arguments in section \ref{sec:1d problem}.

\subsection*{Data availability statement}
Data sharing is not applicable to this article as no new data were created or analyzed in this study.

\end{document}